\documentclass{amcjoucc}

\usepackage{amsmath, amsfonts, amsthm, amssymb, graphicx, dsfont}
\usepackage[numbers,square, comma, sort&compress]{natbib}

\usepackage[small,bf,hang]{caption} % Improved caption

\usepackage{subfig}                      % Figures which consist of multiple subfigures

\usepackage{hyperref}

\usepackage{algorithm, algorithmic}

\usepackage{array}

\usepackage{multirow}

\usepackage{enumitem}

\newcolumntype{M}[1]{>{\centering\arraybackslash}m{#1}}
\newcolumntype{N}{@{}m{0pt}@{}}

\usepackage{color}

\graphicspath{{figures/}}

\newenvironment{dedication}
{\begin{quotation}\begin{center}\begin{em}}
{\end{em}\end{center}\end{quotation}}

\begin{document}

\begin{frontmatter}   %%  Title, information about author, abstract, etc.

\titledata{Improved bounds for hypohamiltonian graphs}{}           % title of the paper
%{Footnote on the title.}                 % footnote on the title -- empty if not required

\authordata{Jan Goedgebeur}            % First author name
{Department of Applied Mathematics, Computer Science \& Statistics, Ghent University, Krijgslaan 281-S9, 9000 Ghent, Belgium}    % Affiliation and address
{jan.goedgebeur@ugent.be}                     % E-mail address
{Both authors are supported by a Postdoctoral Fellowship of the Research Foundation Flanders (FWO).}          % Footnote on the first author (grant number, thanks,
                                         % web page, etc.) -- empty in not required

\authordata{Carol T.\ Zamfirescu}            % Second author
{Department of Applied Mathematics, Computer Science \& Statistics, Ghent University, Krijgslaan 281-S9, 9000 Ghent, Belgium}    % Affiliation and address
{czamfirescu@gmail.com}
{}                                       % No footnote!

\keywords{Hamiltonian, hypohamiltonian, planar, girth, cubic graph, exhaustive generation.}               % Keywords
\msc{05C10, 05C38, 05C45, 05C85.}                       % Math. Subj. Class. codes

\begin{dedication}
In loving memory of Ella.
\end{dedication}

\begin{abstract}
A graph $G$ is \emph{hypohamiltonian} if $G$ is non-hamiltonian and $G - v$ is hamiltonian for every $v \in V(G)$. In the following, every graph is assumed to be hypohamiltonian. Aldred, Wormald, and McKay gave a list of all graphs of order at most 17. In this article, we present an algorithm to generate all graphs of a given order and apply it to prove that there exist exactly 14~graphs of order~18 and 34~graphs of order~19. We also extend their results in the cubic case. Furthermore, we show that (i)~the smallest graph of girth~6 has order~25, (ii)~the smallest planar graph has order at least~23, (iii)~the smallest cubic planar graph has order at least~54, and (iv)~the smallest cubic planar graph of girth~5 with non-trivial automorphism group has order~78.

\bigskip

\noindent
\textbf{Important note:} the version of this manuscript which was published in~\cite{GZ} contained an error in the definition of \textit{good $A$-edge} on page 239. We fixed this error in the current manuscript and also adjusted the code of our generator for hypohamiltonian graphs accordingly~\cite{genhypo-site}. Because of this error, theoretically some hypohamiltonian graphs might have been missed by the program. However, we reran all computations which were reported in~\cite{GZ} and no graphs were missed by the old version of the program.
\end{abstract}

\end{frontmatter}   %% End of the front matter

%% Your article

\section{Introduction}
\label{section:intro}

Throughout this paper all graphs are undirected, finite, connected, and neither contain loops nor multiple edges, unless explicitly stated otherwise. A graph is \emph{hamiltonian} if it contains a cycle visiting every vertex of the graph. Such a cycle or path is called \emph{hamiltonian}. A graph $G$ is \emph{hypohamiltonian} if $G$ is non-hamiltonian, and for every $v \in V(G)$ the graph $G - v$ is hamiltonian.

We call a vertex \emph{cubic} if it has degree~3, and a graph \emph{cubic} if all of its vertices are cubic. Let $G$ be a graph. We use $\deg(v)$ to denote the degree of a vertex~$v$ and $\Delta(G) = \max_{v \in V(G)} \deg(v)$. The \emph{girth} of a graph is the length of its shortest cycle. A cycle of length~$k$ will be called a \emph{$k$-cycle}. For $S \subset V(G)$, $G[S]$ shall denote the graph induced by $S$. A subgraph $G' = (V',E') \subset G = (V,E)$ is \emph{spanning} if $V' = V$. For a set $X$, we denote by $|X|$ its cardinality. We refer to~\cite{Di10} for undefined notions.

The study of hypohamiltonian graphs was initiated in the early sixties by Sousselier~\cite{So63}, and Thomassen made numerous important contributions~\citep{Th74-1,Th74-2,Th76,Th78,Th81}; for further details, see the survey of Holton and Sheehan~\cite{HS93} from 1993. For more recent results and new references not contained in the survey, we refer to the article of Jooyandeh, McKay, \"{O}sterg{\aa}rd, Pettersson, and the second author~\cite{JMOPZ}.

In 1973, Chv\'{a}tal showed~\cite{Ch73} that if we choose $n$ to be sufficiently large, then there exists a hypohamiltonian graph of order~$n$. We now know that for every~$n \ge 18$ there exists such a graph of order $n$, and that 18 is optimal, since Aldred, McKay, and Wormald showed that there is no hypohamiltonian graph on 17~vertices~\cite{AMW97}. Their paper fully settled the question for which orders hypohamiltonian graphs exist and for which they do not exist. For more details, see~\cite{HS93}.

They also provide a complete list of hypohamiltonian graphs with at most 17~vertices. There are seven such graphs: exactly one for each of the orders~10 (the Petersen graph), 13, and 15, four of order~16 (among them Sousselier's graph), and none of order~17. Aldred, McKay, and Wormald~\cite{AMW97} showed that there exist at least thirteen hypohamiltonian graphs with 18 vertices, but the exact number was unknown. In~\cite{mckay-site}, McKay lists all known hypohamiltonian graphs up to 26~vertices (recall that the lists with 18 or more vertices may be incomplete). He also lists all cubic hypohamiltonian graphs up to 26~vertices as well as the cubic hypohamiltonian graphs with girth at least~5 and girth at least~6 on 28 and 30~vertices, respectively. In Section~\ref{sect:results_algo} we extend the results both for the general and cubic case.

The main contributions of this manuscript are: (i)~an algorithm ${\mathfrak A}$ to generate all pairwise non-isomorphic hypohamiltonian graphs of a given order, (ii)~the results of applying this algorithm, and (iii)~an up-to-date overview of the best currently available lower and upper bounds on the order of the smallest hypohamiltonian graphs satisfying various additional properties, see Table~\ref{table:bounds_hypoham_graphs}. The algorithm ${\mathfrak A}$ is based on the algorithm of Aldred, McKay, and Wormald from~\cite{AMW97}, but is extended with several additional bounding criteria which speed it up substantially. Furthermore, ${\mathfrak A}$ also allows to generate planar hypohamiltonian graphs and hypohamiltonian graphs with a given lower bound on the girth far more efficiently.

We present ${\mathfrak A}$ in Section~\ref{section:hypoham_graphs} and showcase the new complete lists of hypohamiltonian graphs we obtained with it. In Section~\ref{section:planar_hypoham_graphs} we illustrate how ${\mathfrak A}$ can be extended to generate planar hypohamiltonian graphs and show how we applied ${\mathfrak A}$ to improve the lower bounds on the order of the smallest planar hypohamiltonian graph. (In the following, unless stated otherwise, when we say that a graph is ``smaller'' or ``the smallest'', we always refer to its order.) Using the program \emph{plantri}~\cite{brinkmann_07}, we also give a new lower bound for the order of the smallest cubic planar hypohamiltonian graph. In an upcoming paper~\cite{GZ2}, we will adapt the approach used in the algorithm ${\mathfrak A}$ to generate \emph{almost hypohamiltonian graphs}~\cite{Za15} efficiently. (A graph $G$ is \emph{almost hypohamiltonian}, if it is non-hamiltonian and there exists a vertex $w$ such that $G - w$ is non-hamiltonian, but $G - v$ is hamiltonian for every vertex $v \ne w$.)

We now discuss the numbers given in Table~\ref{table:bounds_hypoham_graphs} and start with the first row. For girth~3, Aldred, McKay, and Wormald~\cite{AMW97} showed that there is no hypohamiltonian graph of girth~3 and order smaller than~18, and Collier and Schmeichel~\cite{CS77} showed already in 1977 that there exists such a graph on 18~vertices. For girth~4, the results of~\cite{AMW97} imply that there is no such graph on fewer than 18~vertices, and the hypohamiltonian graph presented in Figure~\ref{fig:hypo_18}~(b) from Section~\ref{sect:results_algo}---this graph was given earlier and independently by McKay~\cite{mckay-site}---provides an example of order~18. The third number is due to the Petersen graph, for which it is well-known that it is the smallest hypohamiltonian graph, see e.g.~\cite{HDV67}. The smallest hypohamiltonian graph of girth~6 was obtained by the application of~${\mathfrak A}$ and is shown in Figure~\ref{fig:hypo_25_g6}.
For girth~7, Coxeter's graph provides the smallest example. Its minimality as well as the new lower bound for girth~8 follows from the application of ${\mathfrak A}$. The bound for girth~9 follows from an argument given at the end of the following paragraph. Note that, as M\'{a}\v{c}ajov\'{a} and \v{S}koviera mention in~\cite{MS11}, no hypohamiltonian graphs of girth greater than~7 are known, and Coxeter's graph is the only known cyclically 7-connected hypohamiltonian graph of girth~7.

Concerning the second row, Thomassen~\cite{Th81} showed that there exists a cubic hypohamiltonian graph of girth~4 and order~24. Petersen's graph is responsible for the second value, Isaacs' flower snark~$J_7$ and Coxeter's graph give the upper bounds for girth~6 and 7, respectively. Through an exhaustive computer-search, McKay was able to determine the order of the smallest cubic hypohamiltonian graph of girth 4,~5,~6,~and~7, establishing that the aforementioned graphs turned out to be the smallest of a fixed girth, see~\cite{mckay-site}. (Note that McKay does not state this explicitly, and that these results were verified independently by the first author.) We obtained the improved lower bounds for girth~8 and~9 through an exhaustive computer-search (see Section~\ref{sect:results_algo} for more details). Now let $G$ be a hypohamiltonian graph of girth~9 containing a non-cubic vertex~$v$. Then $\{ w \in V(G) : d(v,w) \le 4 \}$, where $d(v,w)$ denotes the number of edges in a shortest path between vertices $v$ and $w$, consists of pairwise different vertices, so $|V(G)| \ge 61$. (Recall that as is shown in Table~\ref{table:bounds_hypoham_graphs}, if $G$ is a cubic hypohamiltonian graph of girth~9, then $|V(G)| \ge 66$.)

In the third row, the first upper bound is due to Thomassen, see~\cite{Th76}, while the second one is due to Jooyandeh, McKay, \"{O}sterg{\aa}rd, Pettersson, and the second author~\cite{JMOPZ}. The previous best lower bounds were provided by~\cite{AMW97}---although that paper does not address planarity---while the current best lower bounds are proven using ${\mathfrak A}$, see Section~\ref{section:planar_hypoham_graphs}. In~\cite{JMOPZ} it was also shown that there exists a planar hypohamiltonian graph of girth~5 on 45~vertices, and that there is no smaller such graph.

The upper bound for the smallest cubic planar hypohamiltonian graph of girth~4 was established by Araya and Wiener~\cite{AW11}. The best available lower bound prior to this paper can be found in the same article~\cite{AW11} and was~44. We improved this to~54 with the program \emph{plantri}~\cite{brinkmann_07} as described in Section~\ref{section:planar_cubic}. Finally, McKay~\cite{Mc} recently proved that the order of the smallest cubic planar hypohamiltonian graph of girth~5 is 76.

In Table~\ref{table:bounds_hypoham_graphs}, we denote by ``--'' an impossible combination of properties. There are two arguments from which these impossibilities follow. Firstly, a cubic hypohamiltonian graph cannot contain triangles, as proven by Collier and Schmeichel~\cite{CS78}. Secondly, it follows from Euler's formula that a planar 3-connected graph---it is easy to see that every hypohamiltonian graph is 3-connected---has girth at most~5.

\begin{table}
\centering
\begin{tabular}{ M{2.4cm} | M{10mm} M{10mm} M{10mm} M{10mm} M{10mm} M{10mm} M{10mm} N}
  girth & 3 & 4  & \ \ 5 \ \  & 6  &  7 & 8 & 9 &\\
  \hline
  general & 18 & 18 & 10 & \ \ \textbf{25} \newline \footnotesize{18..28} & \ \ \textbf{28} \newline \footnotesize{18..28} & \textbf{36}..$\infty$ \newline \footnotesize{18..$\infty$} & \textbf{61}..$\infty$ \newline \footnotesize{18..$\infty$} &\\[15pt]
  cubic & -- & 24 & 10 & 28 & 28 & \textbf{50}..$\infty$ \newline \footnotesize{30..$\infty$} & \textbf{66}..$\infty$ \newline \footnotesize{58..$\infty$} &\\[15pt]
  planar & \textbf{23}..240 \newline \footnotesize{18..240} & \textbf{25}..40 \newline \footnotesize{18..40} & 45 & -- & -- & -- & -- &\\[15pt]
  planar and cubic & -- & \textbf{54}..70 \newline \footnotesize{44..70} & 76 & -- & -- & -- & -- &\\
\end{tabular}
\caption{Bounds for the order of the smallest hypohamiltonian graph with additional properties. The bold numbers are new bounds obtained in this manuscript; if an entry contains two lines, the upper line indicates the new bounds, while the lower line shows the previous bounds. The symbol ``--'' designates an impossible combination of properties and $a..b$ means that the number is at least $a$ and at most $b$. $b = \infty$ signifies that no graph with the given properties is known.}
\label{table:bounds_hypoham_graphs}
\end{table}

\section{Generating hypohamiltonian graphs}
\label{section:hypoham_graphs}

\subsection{Preparation}

In this section we present our algorithm ${\mathfrak A}$ to generate all non-isomorphic hypohamiltonian graphs of a given order. ${\mathfrak A}$ is based on work of Aldred, McKay, and Wormald~\cite{AMW97}, but contains essential additional bounding criteria. It is easy to see that hypohamiltonian graphs are 3-connected and cyclically 4-connected. 

We follow Aldred, McKay and Wormald~\cite{AMW97} and say that a graph $G$ is \emph{hypocyclic} if for every $v \in V(G)$, the graph~$G - v$ is hamiltonian. Hamiltonian hypocyclic graphs are usually called ``1-hamiltonian'' (see e.g.~\cite{CKL70}), so the family of all hypocyclic graphs is the disjoint union of the families of all 1-hamiltonian and hypohamiltonian graphs.

We now present several lemmas with necessary conditions for a graph to be hypocyclic or hypohamiltonian. We then use a selection of these lemmas to prune the search in the generation algorithm. This selection, i.e.\ whether to use a certain lemma or not and the order in which these lemmas should be applied, is based on experimental evidence. The efficiency of the algorithm strongly depends on the strength of these pruning criteria.

To avoid confusion, we will generally use the same terminology as Aldred, McKay, and Wormald did in~\cite{AMW97} (that is: e.g.\ type A, B, and C obstructions). Let $G$ be a possibly disconnected graph. We will denote by $p(G)$ the minimum number of disjoint paths needed to cover all vertices of $G$, by $V_1(G)$ the vertices of degree~1 in $G$, 
%and by $I(G)$ the set of all isolated vertices and all isolated edges of~$G$. Put
and by $I(G)$ the set of all isolated vertices and all isolated $K_2$'s (i.e.\ isolated edges together with their endpoints) of~$G$. Put
$$k(G) = \begin{cases}
                            0 & {\rm if } \ G \ {\rm is \ empty,}\\
                            \max \left\{ 1, \left\lceil{\frac{|V_1|}{2}}\right\rceil \right\}  & {\rm if } \ I(G) = \emptyset \ {\rm but} \ G \ {\rm is \ not \ empty},\\
                            |I(G)| + k(G - I(G)) & {\rm else}.
                        \end{cases}$$

\begin{lemma}[Aldred, McKay, and Wormald~\cite{AMW97}]\label{lem:typeA+B_obstr}
Given a hypocyclic graph $G$, for any partition $(W,X)$ of the vertices of $G$ with $|W| > 1$ and $|X| > 1$, we have that $$p(G[W]) < |X| \quad {\rm {\it and}} \quad k(G[W]) < |X|.$$
\end{lemma}

Now consider a graph $G$ containing a partition $(W,X)$ of its vertices with $|W| > 1$ and $|X| > 1$. If $p(G[W]) \ge |X|$, then we call $(W,X)$ a \emph{type A obstruction}, and if $k(G[W]) \ge |X|$, then we speak of a \emph{type B obstruction}. For efficiency reasons we only consider type A obstructions where $G[W]$ is a union of disjoint paths.

\begin{lemma}[Aldred, McKay, and Wormald~\cite{AMW97}]\label{lem:typeC_obstr}
Let $G$ be a hypocyclic graph, and consider a partition $(W,X)$ of the vertices of $G$ with $|W| > 1$ and $|X| > 1$ such that $W$ is an independent set. Furthermore, for some
vertex $v \in X$, define $n_1$ and $n_2$ to be the number of vertices of $X - v$ joined to one or more than one vertex of $W$, respectively. Then we have $2n_2 + n_1 \ge 2 |W|$ for every $v \in X$.
\end{lemma}

If all assumptions of Lemma~\ref{lem:typeC_obstr} are met and $2n_2 + n_1 < 2 |W|$ for some $v \in X$, we call $(W,X,v)$ a \emph{type C obstruction}.

Intuitively, by a \emph{good $Y$-edge} (for $Y \in \{ A, B, C \}$) we mean an edge which works towards the destruction of a type $Y$~obstruction. We will now formally define these good $Y$-edges.

We use Lemma~\ref{lem:typeA+B_obstr} as follows. Assume $G'$ is a hypohamiltonian graph and that $G$ is a spanning subgraph of $G'$ which contains a type A obstruction $(W,X)$ (we choose $W$ such that $G[W]$ is a union of disjoint paths). %Since $G'$ is hypohamiltonian it cannot contain a type A obstruction, so there must be an edge $e$ in $E(G') \setminus E(G)$ whose endpoints are in different components of $G[W]$ and for which at least one of the endpoints has degree at most one in $G[W]$. We call such an edge a \textit{good $A$-edge} for $(W,X)$.
Since $G'$ is hypohamiltonian it cannot contain a type A obstruction, so there must be an edge in $E(G') \setminus E(G)$ whose endpoints are in different components of $G[W]$. We call such an edge a \textit{good $A$-edge} for $(W,X)$. Note that there must also be an edge in $E(G') \setminus E(G)$ for which at least one of the endpoints has degree at most one in $G[W]$ (but its endpoints could be in the same component of $G[W]$). Our computational experiments indicate that when generating graphs with girth at least 4, it is more efficient to define a \textit{good $A$-edge} for $(W,X)$ as an edge in $E(G') \setminus E(G)$ for which at least one of the endpoints has degree at most one in $G[W]$.
To clarify our procedure: let $\bar G$ be the graph obtained after adding a good $A$-edge to $G$. If $\bar G[W]$ contains vertices of degree 3, we remove those vertices from $W$ (and add them to $X$) for the next iteration of the algorithm in order to guarantee that $\bar G[W]$ remains a union of disjoint paths.

%Aldred, McKay, and Wormald~\cite{AMW97} did use this obstruction, but they did not require these good A-edges to have an endpoint of degree at most one in $G[W]$ (which turns out to be far more restrictive). 
Similarly, a \textit{good $B$-edge} for a type B obstruction $(W,X)$ in $G$ is a non-edge of $G$ that joins two vertices of $W$ where at least one of those vertices has degree at most one in $G[W]$. Finally, a \textit{good $C$-edge} for a type C obstruction $(W,X,v)$ in $G$ is a non-edge $e$ of $G$ for which one of the two following conditions holds:

\begin{itemize}
\item[(i)] Both endpoints of $e$ are in $W$.

\item[(ii)] One endpoint of $e$ is in $W$ and the other endpoint is in $X-v$ and has at most one neighbour in $W$.
\end{itemize}

We leave the straightforward verification that this is the only way to destroy a type B/C obstruction to the reader. Likewise, it is elementary to see that every hypohamiltonian graph has minimum degree~3---we are mentioning this explicitly, since we will later make use of the fact that hypohamiltonian graphs do not contain vertices of degree~2---, and that it is not bipartite. However, for every $k \ge 23$ there exists a hypohamiltonian graph containing the complete bipartite graph $K_{2k - 44, 2k - 44}$, as proven by Thomassen~\cite{Th81}.

\begin{lemma}[Collier and Schmeichel~\cite{CS78}]\label{lem:triangle_obstr}
Let $G$ be a hypohamiltonian graph containing a triangle $T$. Then every vertex of $T$ has degree at least~$4$.
\end{lemma}

A \emph{diamond} is a $K_4$ minus an edge and the \textit{central edge} of a diamond is the edge between the two cubic vertices.

\begin{proposition}\label{prop:diamond_obstr}
Let $G$ be a hypohamiltonian graph containing a diamond with vertices $a,b,c,d$ and central edge $ac$. Then the degrees of $a$ and $c$ (in $G$) are at least~$5$.
\end{proposition}
\begin{proof}
It follows from Lemma~\ref{lem:triangle_obstr} that $a$ is not cubic. Let $a$ have degree~4. Since $G$ is hypocyclic, $G - c$ contains a hamiltonian cycle ${\mathfrak h}$. ${\mathfrak h}$ must contain $ab$ or $ad$ (possibly both), say $ab$. But then $({\mathfrak h} - ab) \cup acb$ is a hamiltonian cycle in $G$, a contradiction.
\end{proof}

Note that in Proposition~\ref{prop:diamond_obstr}, the edge $bd$ may or may not be present in the graph. We have already mentioned that hypohamiltonian graphs are cyclically 4-connected. We can strengthen this in the following way.

\begin{lemma}\label{lem:edge-cut}
One of the two components obtained when deleting a $3$-edge-cut from a hypohamiltonian graph must be $K_1$.
\end{lemma}
\begin{proof}
Consider a 3-edge-cut $C$ in a hypohamiltonian graph $G$. $G - C$ has two components $A$ and $B$ with $|V(A)| \le |V(B)|$. We put $C = \{ a_1b_1, a_2b_2, a_3b_3 \}$, where $a_i \in V(A)$ and $b_i \in V(B)$. Assume $A \ne K_1$. In this situation, since $G$ is 3-connected, the elements of the set $\{ a_1, a_2, a_3, b_1, b_2, b_3 \}$ are pairwise distinct, as otherwise we would have a 2-cut.

Since $G$ is hypohamiltonian, $G - b_3$ is hamiltonian, so there is a hamiltonian path ${\mathfrak p}_A$ in $A$ with end-vertices $a_1$ and $a_2$. As $G - a_3$ is hamiltonian, there is a hamiltonian path ${\mathfrak p}_B$ in $B$ with end-vertices $b_1$ and $b_2$. Now ${\mathfrak p}_A \cup {\mathfrak p}_B + a_1b_1 + a_2b_2$ is a hamiltonian cycle in $G$, a contradiction.
\end{proof}

\begin{proposition}\label{prop:vertex-cut}
Let $G$ be a hypohamiltonian graph containing a $3$-cut $M = \{ u,v,w \}$.
\begin{enumerate}[label=(\roman*)]
\item We have $uv,vw,wu \notin E(G)$.
\item If $M$ is not the neighbourhood of a vertex, then $\max_{x \in M} \deg(x) \ge 4$.
\end{enumerate}
\end{proposition}

\begin{proof}
(i) Note that (i) was also already shown by Thomassen in~\cite{Th76}, but here we give an alternative proof. Assume that $uv \in E(G)$. Since $G$ is hypohamiltonian, there exists a hamiltonian cycle ${\mathfrak h}$ in $G - u$. Let $A$ and $B$ be the components of $G - M$ (we leave to the reader the easy proof that there are exactly two components in $G - M$) and put ${\mathfrak p}_A = {\mathfrak h} \cap G[V(A) \cup M]$.

Case 1: $A \ne K_1$ and $B \ne K_1$. Since $M$ is a 3-cut, ${\mathfrak p}_A$ has end-vertices $v$ and $w$. Analogously there exists a hamiltonian path ${\mathfrak p}_B$ in $G[V(B) \cup M]$ with end-vertices $u$ and $w$. Now ${\mathfrak p}_A \cup {\mathfrak p}_B + uv$ is a hamiltonian cycle in $G$, a contradiction.

Case 2: $A = K_1$. We have $V(A) = \{ a \}$, so $M = N(a)$. Now $auv$ is a triangle containing the cubic vertex $a$, in contradiction to Lemma~\ref{lem:triangle_obstr}.

(ii) follows directly from Lemma~\ref{lem:edge-cut}. Note that the neighbourhood condition is necessary, since cubic hypohamiltonian graphs---such as the Petersen graph---do exist.
\end{proof}

\begin{corollary}
In a cubic hypohamiltonian graph, every $3$-cut must be the neighbourhood of a vertex.
\end{corollary}

\subsection{The enumeration algorithm}

The pseudocode of the enumeration algorithm ${\mathfrak A}$ is given in Algorithm~\ref{algo:init-algo} and Algorithm~\ref{algo:construct}.

In order to generate all hypohamiltonian graphs with $n$ vertices we start from a graph $G$ which consists of an $(n-1)$-cycle and an isolated vertex $h$ (disjoint from the cycle), so $G - h$ is hamiltonian. Both in Algorithm~\ref{algo:init-algo} and Algorithm~\ref{algo:construct} we only add edges between existing vertices of the graph. So if a graph is hamiltonian, all graphs obtained from it will also be hamiltonian. Thus we can prune the search when a hamiltonian graph is constructed (cf.\ line~\ref{line:hamiltonian} of Algorithm~\ref{algo:construct}).

In Algorithm~\ref{algo:init-algo} we connect $h$ to $D$ vertices of the $(n-1)$-cycle in all possible ways and then perform Algorithm~\ref{algo:construct} on these graphs which will continue to recursively add edges without increasing the maximum degree of the graph.

It is essential for the efficiency of the algorithm that as few as possible edges are added (i.e.\ that as few as possible graphs are constructed), while still guaranteeing that all hypohamiltonian graphs are found by the algorithm. If a generated graph contains an obstruction for hypohamiltonicity, it clearly cannot be hypohamiltonian and hence we only add edges which destroy (or work towards the destruction of) that obstruction.

In the following theorem we show that this algorithm indeed finds all hypohamiltonian graphs.

\begin{theorem}\label{lem:diamond-algo-correct}
If Algorithm~\ref{algo:init-algo} terminates, the list of graphs $\mathcal H$ outputted by the algorithm is the list of all hypohamiltonian graphs with $n$ vertices.
\end{theorem}
\begin{proof}

It follows from line~\ref{line:hypohamiltonian} of Algorithm~\ref{algo:construct} that $\mathcal H$ only contains hypohamiltonian graphs. Now we will show that $\mathcal H$ indeed contains \emph{all} hypohamiltonian graphs with $n$~vertices.

Consider a hypohamiltonian graph $G$ with $n$ vertices. It follows from the definition of hypohamiltonicity that there is a spanning subgraph $G_0$ of $G$ which consists of an $(n-1)$-cycle $C$ and a vertex $v$ disjoint from $C$ which is connected to $\Delta(G)$ vertices of $C$. Since Algorithm~\ref{algo:init-algo} connects the vertex $h$ with $D$ vertices of an $(n-1)$-cycle in all possible ways, it will also construct a graph which is isomorphic to $G_0$.

We will now show by induction that Algorithm~\ref{algo:construct} produces a graph isomorphic to a spanning subgraph $G$ with $i$ edges for every $|E(G_0)| \le i \le |E(G)|$.

Assume this claim holds for some $i$ with $|E(G_0)| \le i \le |E(G)|-1$ and call the graph produced by Algorithm~\ref{algo:construct} which is isomorphic to a spanning subgraph of $G$ with $i$ edges $G'$.

Assume that $G'$ contains a type A obstruction $(W,X)$. By Lemma~\ref{lem:typeA+B_obstr}, $G$ does not contain a type A obstruction, so there is a good A-edge $e$ for $(W,X)$ in $E(G) \setminus E(G')$. It follows from line~\ref{line:destroy_typeA} of Algorithm~\ref{algo:construct} that Construct($G' + e, D$) is called and $G' + e$ will be accepted by the algorithm since $G$ is non-hamiltonian.

We omit the discussion of the cases where $G'$ contains a type B or C obstruction (i.e.\ lines~\ref{line:typeB_obstr} and~\ref{line:typeC_obstr}, respectively) as this is completely analogous.

So assume that $G'$ does not contain a type A obstruction, but contains a vertex $v$ of degree two (note that $G'$ cannot contain vertices of degree less than two). Since a hypohamiltonian graph has minimum degree 3, there is an edge $e \in E(G) \setminus E(G')$ which contains $v$ as an endpoint. It follows from line~\ref{line:destroy_deg2} of Algorithm~\ref{algo:construct} that Construct($G' + e, D$) is called.

The case where $G'$ contains a cubic vertex which is part of a triangle (i.e.\ line~\ref{line:deg3_triangle}) is completely analogous.

If none of the criteria is applicable, Algorithm~\ref{algo:construct} adds an edge $e$ to $G'$ in all possible ways (without increasing the maximum degree) and calls Construct($G'+e, D$) for each $e$. Since $|E(G')| < |E(G)|$, at least one of the graphs $G'+e$ will be a spanning subgraph of $G$ with $i+1$ edges.

\end{proof}

\begin{algorithm}[h]
\caption{Generate all hypohamiltonian graphs with $n$ vertices}
\label{algo:init-algo}
  \begin{algorithmic}[1]
	\STATE let $\mathcal H$ be an empty list
	\STATE let $G := C_{n-1} + h$
	\FORALL{$3 \le D \le n-1$}
		\STATE // Generate all hypohamiltonian graphs with $\Delta = D$
		\FOR{every way of connecting $h$ of $G$ with $D$ vertices of the $C_{n-1}$}
		\STATE Call the resulting graph $G'$
		\STATE Construct($G', D$) \ // i.e.\ perform Algorithm~\ref{algo:construct}
		\ENDFOR
	\ENDFOR
	\STATE Output $\mathcal H$
  \end{algorithmic}
\end{algorithm}

\begin{algorithm}[ht!]
\caption{Construct(Graph $G$, int $D$)}
\label{algo:construct}
  \begin{algorithmic}[1]
		\IF{$G$ is non-hamiltonian AND not generated before} \label{line:hamiltonian}
			\IF{$G$ contains a type A obstruction $(W,X)$}
				\FOR{every good A-edge $e \notin E(G)$ for $(W,X)$ for which $\Delta(G+e)=D$}
					\STATE Construct($G+e, D$) \label{line:destroy_typeA}
				\ENDFOR
			\ELSIF{$G$ contains a vertex $v$ of degree 2}
				\FOR{every edge $e \notin E(G)$ which contains $v$ as an endpoint for which $\Delta(G+e)=D$}
					\STATE Construct($G+e, D$) \label{line:destroy_deg2}
				\ENDFOR
			\ELSIF{$G$ contains a type C obstruction $(W,X,v)$} \label{line:typeC_obstr}
				\FOR{every good C-edge $e \notin E(G)$ for $(W,X,v)$ for which $\Delta(G+e)=D$}
					\STATE Construct($G+e, D$)
				\ENDFOR
			\ELSIF{$G$ contains a vertex $v$ of degree 3 which is part of a triangle} \label{line:deg3_triangle}
				\FOR{every edge $e \notin E(G)$ which contains $v$ as an endpoint for which \\\qquad $\Delta(G+e)=D$}
					\STATE Construct($G+e, D$)
				\ENDFOR			
			\ELSIF{$G$ contains a type B obstruction $(W,X)$} \label{line:typeB_obstr}
				\FOR{every good B-edge $e \notin E(G)$ for $(W,X)$ for which $\Delta(G+e)=D$}
					\STATE Construct($G+e, D$)
				\ENDFOR					
			\ELSE
				\IF{$G$ is hypohamiltonian} \label{line:hypohamiltonian}
					\STATE add $G$ to the list $\mathcal H$
				\ENDIF
					\FOR{every edge $e \notin E(G)$ for which $\Delta(G+e)=D$}
						\STATE Construct($G+e, D$)
					\ENDFOR	
			\ENDIF
		\ENDIF	
  \end{algorithmic}
\end{algorithm}

To make sure no isomorphic graphs are accepted, we use the program \textit{nauty}~\cite{nauty-website, mckay_14}. In principle more sophisticated isomorphism rejection techniques are known (such as the canonical construction path method~\cite{mckay_98}), but these methods are not compatible with the destruction of obstructions for hypohamiltonicity. Furthermore, isomorphism rejection is not a bottleneck in our implementation of this algorithm.

Also note that we only have to perform the hypohamiltonicity test (which can be computationally very expensive) if the graph does not contain any obstructions for hypohamiltonicity (cf.\ line~\ref{line:hypohamiltonian} of Algorithm~\ref{algo:construct}). Therefore, the hypohamiltonicity test is not a bottleneck in the algorithm.

Since our algorithm only adds edges and never removes any vertices or edges, all graphs obtained by the algorithm from a graph with a $g$-cycle will have a cycle of length at most~$g$. So in case we only want to generate hypohamiltonian graphs with a given lower bound $k$ on the girth, we can prune the construction when a graph with a cycle with length less than $k$ is constructed.

The order in which the bounding criteria of Algorithm~\ref{algo:construct} are tested is vital for the efficiency of the algorithm. By performing various extensive experiments, it turned out that the order in which the bounding criteria are listed in Algorithm~\ref{algo:construct} is the most efficient.

We also note that even though Aldred, McKay, and Wormald mentioned type C obstructions in their paper~\cite{AMW97}, they did not use them in their algorithm. However, our experimental results show that type C obstructions are significantly more helpful than e.g.\ type B obstructions.

%\subsection{Testing and results}

\subsection{Results}
\label{sect:results_algo}

\subsubsection{The general case}

We implemented the algorithm ${\mathfrak A}$ in C and used it to generate all pairwise non-isomorphic hypohamiltonian graphs of a given order (with a given lower bound on the girth). Our implementation of this algorithm is called \textit{GenHypohamiltonian}, and can be downloaded from~\cite{genhypo-site}.

Table~\ref{table:number_of_hypoham_graphs} shows the counts of the complete lists hypohamiltonian graphs which were generated by our program. We generated all hypohamiltonian graphs up to 19 vertices and also went several steps further for hypohamiltonian graphs with a given lower bound on the girth. Recall that previously the complete lists of hypohamiltonian graphs were only known up to 17~vertices. For more information about the previous bounds and results, we refer to Table~\ref{table:bounds_hypoham_graphs} from Section~\ref{section:intro}.

In~\cite{AMW97} Aldred, McKay, and Wormald also produced a sample of 13~hypohamiltonian graphs with 18~vertices. It follows from our results that there are exactly 14 hypohamiltonian graphs with 18 vertices. These graphs are shown in Figure~\ref{fig:hypo_18}. The fourteenth graph which was not already known has girth~5 and is shown in Figure~\ref{fig:hypo_18}~(n). It has automorphism group size 36 and it has the largest group size among the hypohamiltonian graphs with 18 vertices. Using ${\mathfrak A}$, we also showed that there are exactly 34 hypohamiltonian graphs with 19 vertices. As can be seen from Table~\ref{table:number_of_hypoham_graphs}, all 34 of them have girth 5.

All graphs from Table~\ref{table:number_of_hypoham_graphs} can also be
downloaded from the \textit{House of Graphs}~\cite{hog} at
\url{http://hog.grinvin.org/Hypohamiltonian} and also be inspected in the database of interesting graphs by searching for the keywords ``hypohamiltonian * 2016".

\begin{table}
\centering
\small

	\begin{tabular}{|c || c | c | c | c | c | c |}
		\hline
		Order & \# hypoham. & $g \geq 4$ & $g \geq 5$ & $g \geq 6$ &  $g \geq 7$ & $g \geq 8$\\
		\hline
$0-9$  &  0  &  0  &  0  &  0  &  0  &  0  \\
10  &  1  &  1  &  1  &  0  &  0  &  0  \\
11  &  0  &  0  &  0  &  0  &  0  &  0  \\
12  &  0  &  0  &  0  &  0  &  0  &  0  \\
13  &  1  &  1  &  1  &  0  &  0  &  0  \\
14  &  0  &  0  &  0  &  0  &  0  &  0  \\
15  &  1  &  1  &  1  &  0  &  0  &  0  \\
16  &  4  &  4  &  4  &  0  &  0  &  0  \\
17  &  0  &  0  &  0  &  0  &  0  &  0  \\
18  &  14  &  13  &  8  &  0  &  0  &  0  \\
19  &  34  &  34  &  34  &  0  &  0  &  0  \\
20  &  ?  &  $\ge 98$  &  4  &  0  &  0  &  0  \\
21  &  ?  &  ?  &  85  &  0  &  0  &  0  \\
22  &  ?  &  ?  &  420  &  0  &  0  &  0  \\
23  &  ?  &  ?  &  85  &  0  &  0  &  0  \\
24  &  ?  &  ?  &  2 530  &  0  &  0  &  0  \\
25  &  ?  &  ?  &  ?  &  1  &  0  &  0  \\
26  &  ?  &  ?  &  ?  &  0  &  0  &  0  \\
27  &  ?  &  ?  &  ?  &  ?  &  0  &  0  \\
28  &  ?  &  ?  &  ?  &  $\ge 2$  &  1  &  0  \\
29  &  ?  &  ?  &  ?  &  ?  &  0  &  0  \\
30  &  ?  &  ?  &  ?  &  ?  &  0  &  0  \\
$31-35$  &  ?  &  ?  &  ?  &  ?  &  ?  &  0  \\
		
		\hline
	\end{tabular}

\caption{The number of hypohamiltonian graphs. The columns with a header of the form $g \geq k$ contain the number of hypohamiltonian graphs with girth at least~$k$. The counts of cases indicated with a '$\ge$' are possibly incomplete; all other cases are complete.}
\label{table:number_of_hypoham_graphs}

\end{table}

\begin{figure}[h!t]
    \centering
    \subfloat[]{\includegraphics[width=0.21\textwidth]{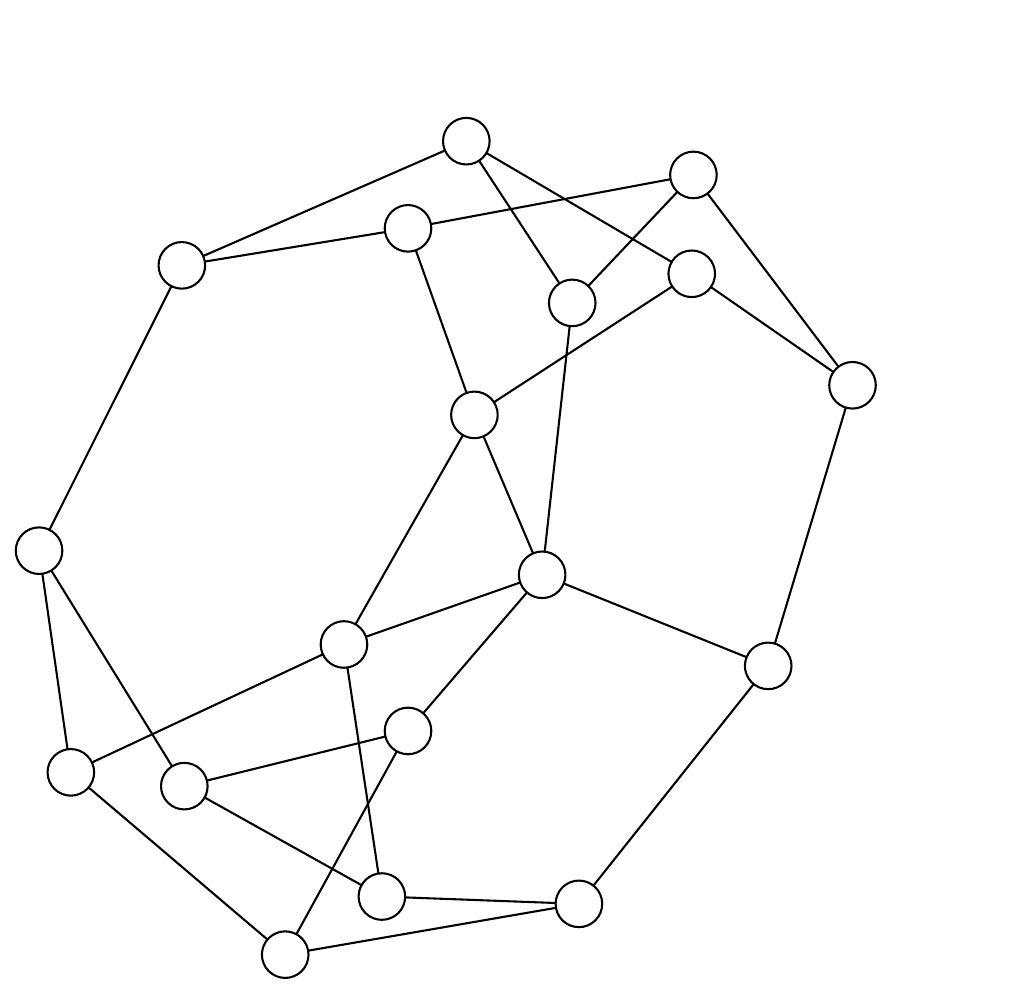}}
    \subfloat[]{\includegraphics[width=0.21\textwidth]{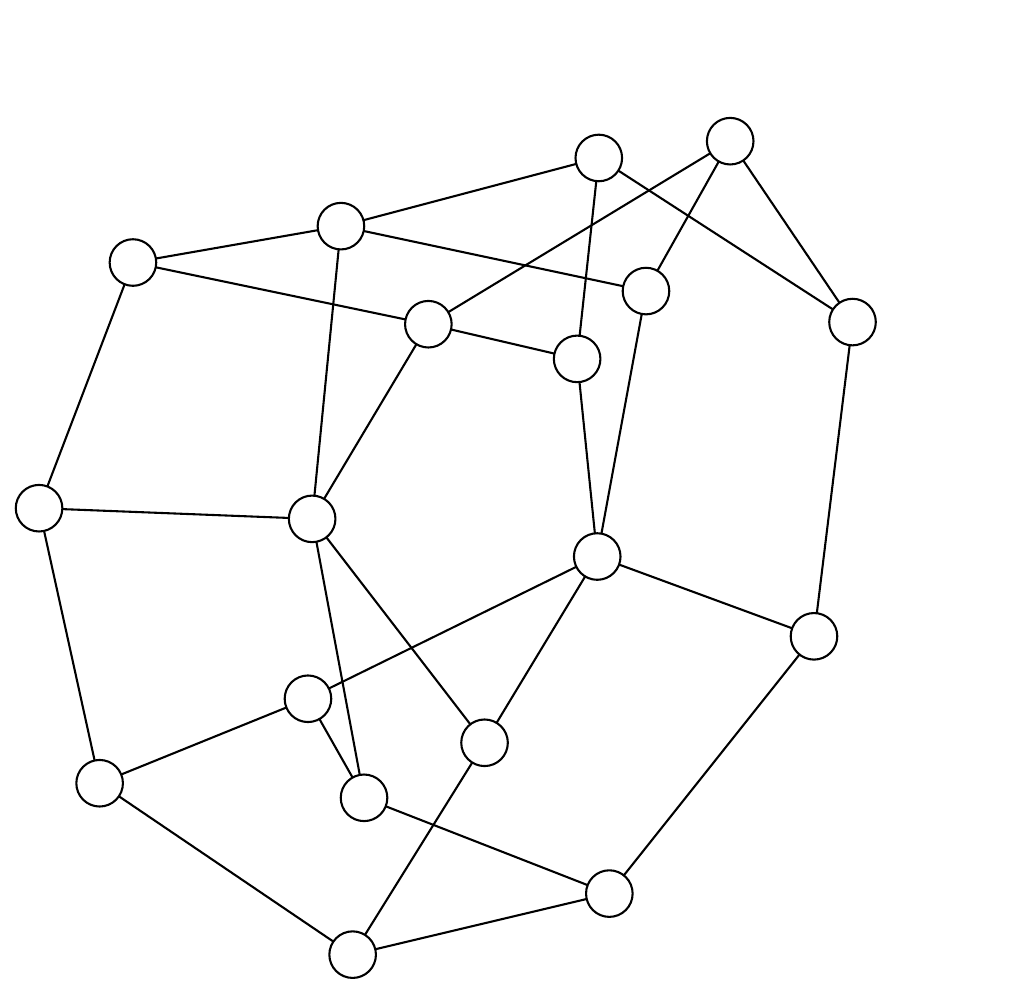}}
    \subfloat[]{\includegraphics[width=0.21\textwidth]{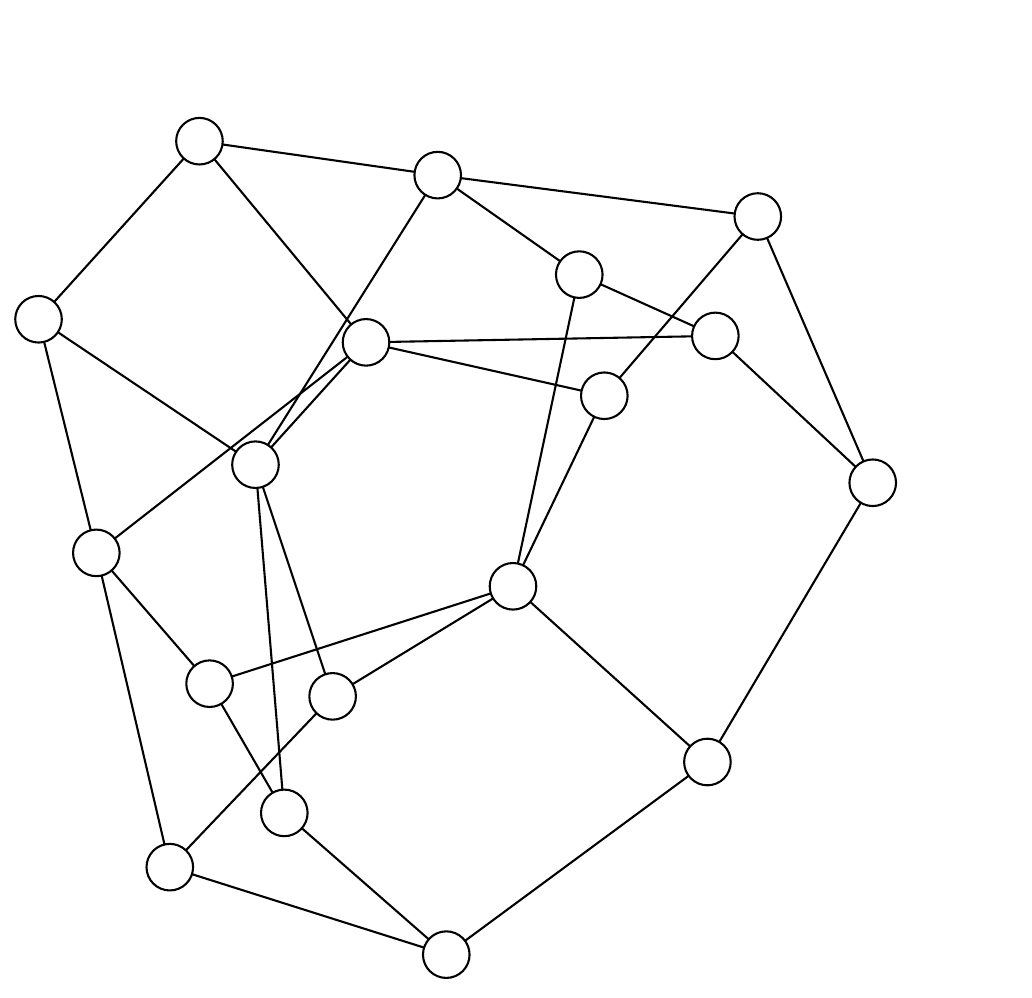}}
    \subfloat[]{\includegraphics[width=0.21\textwidth]{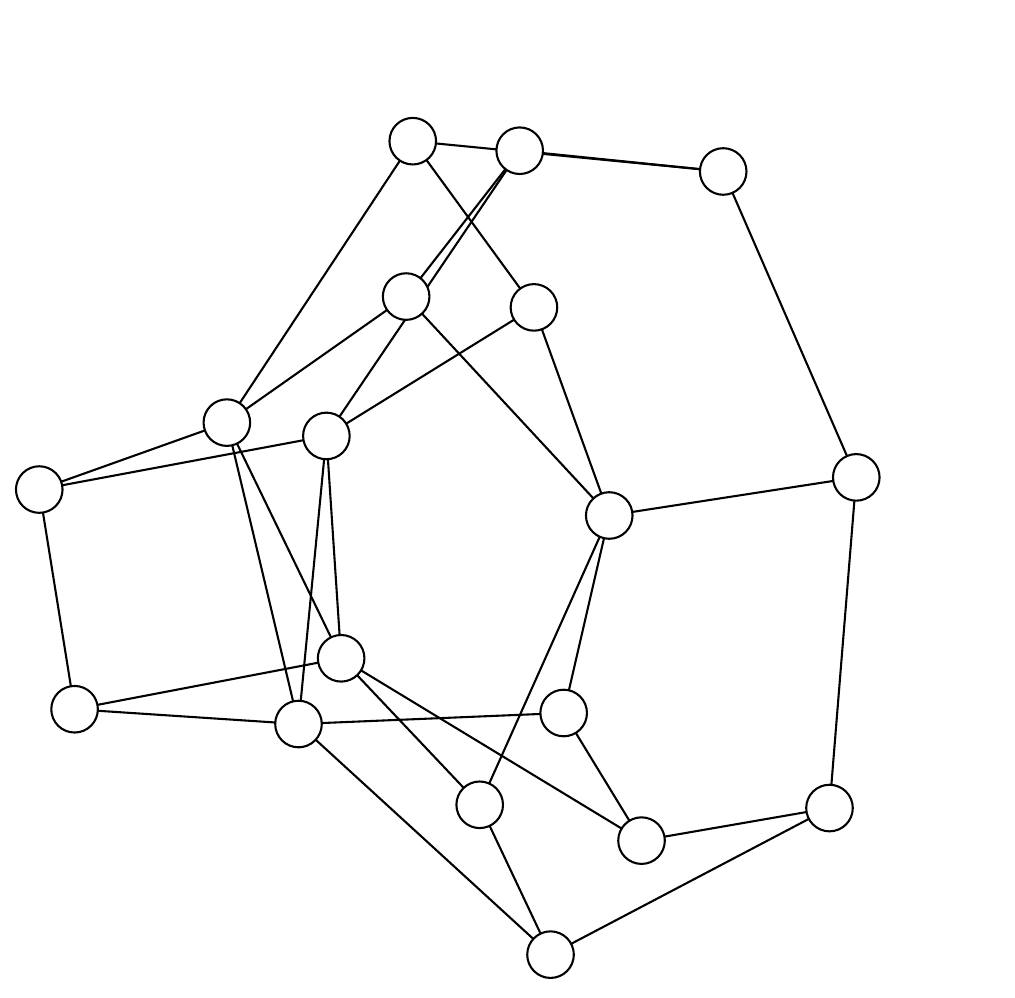}}
    \subfloat[]{\includegraphics[width=0.17\textwidth]{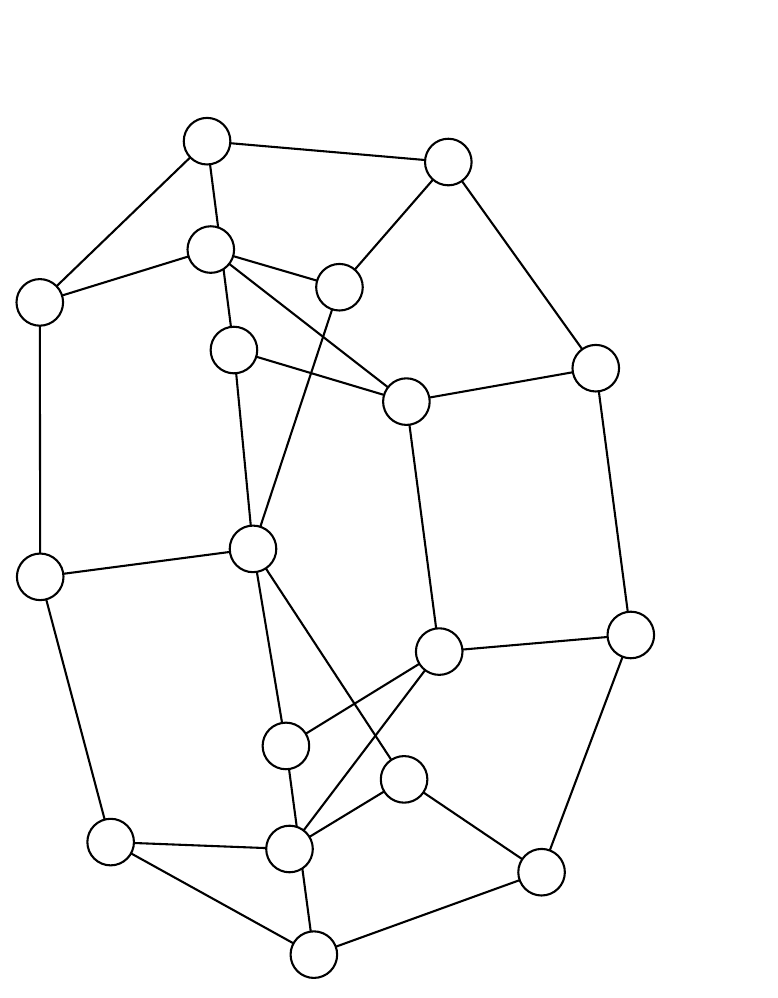}}\\
    \subfloat[]{\includegraphics[width=0.21\textwidth]{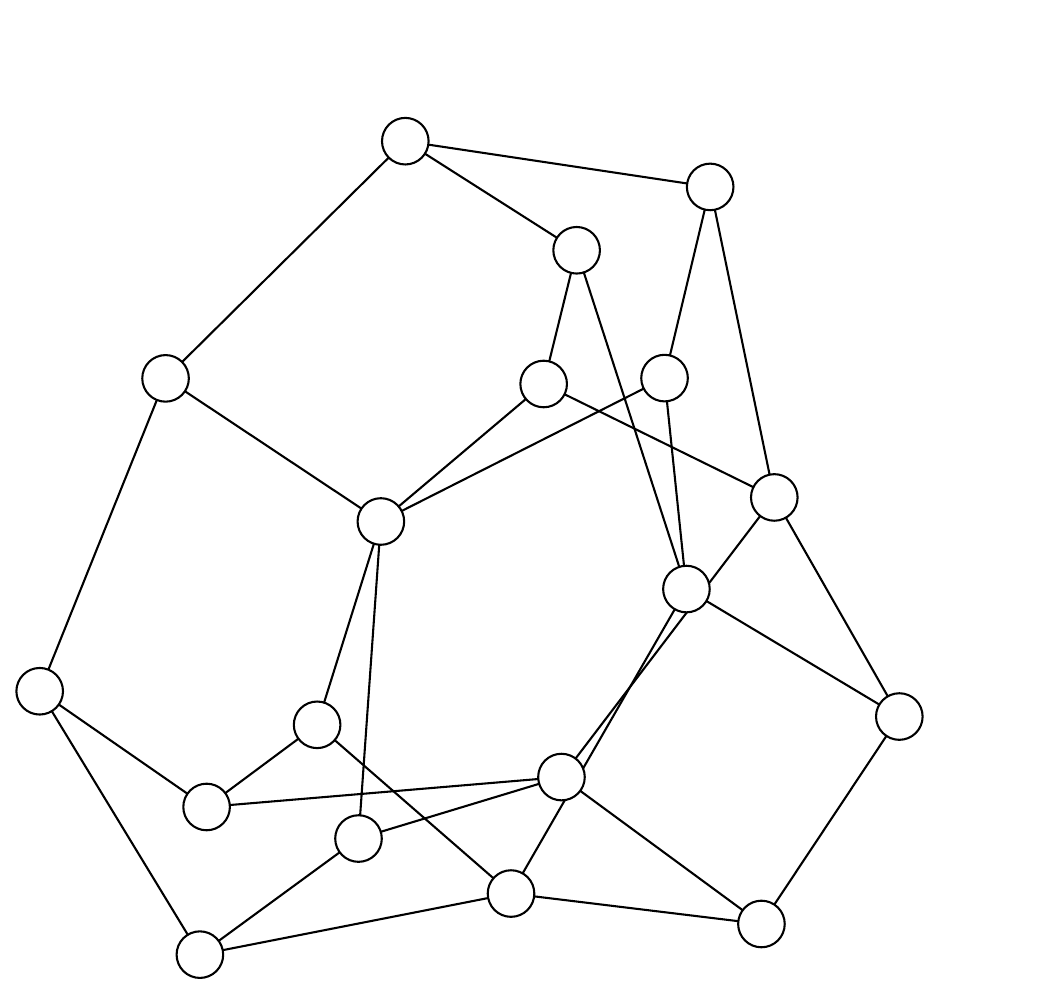}}
    \subfloat[]{\includegraphics[width=0.19\textwidth]{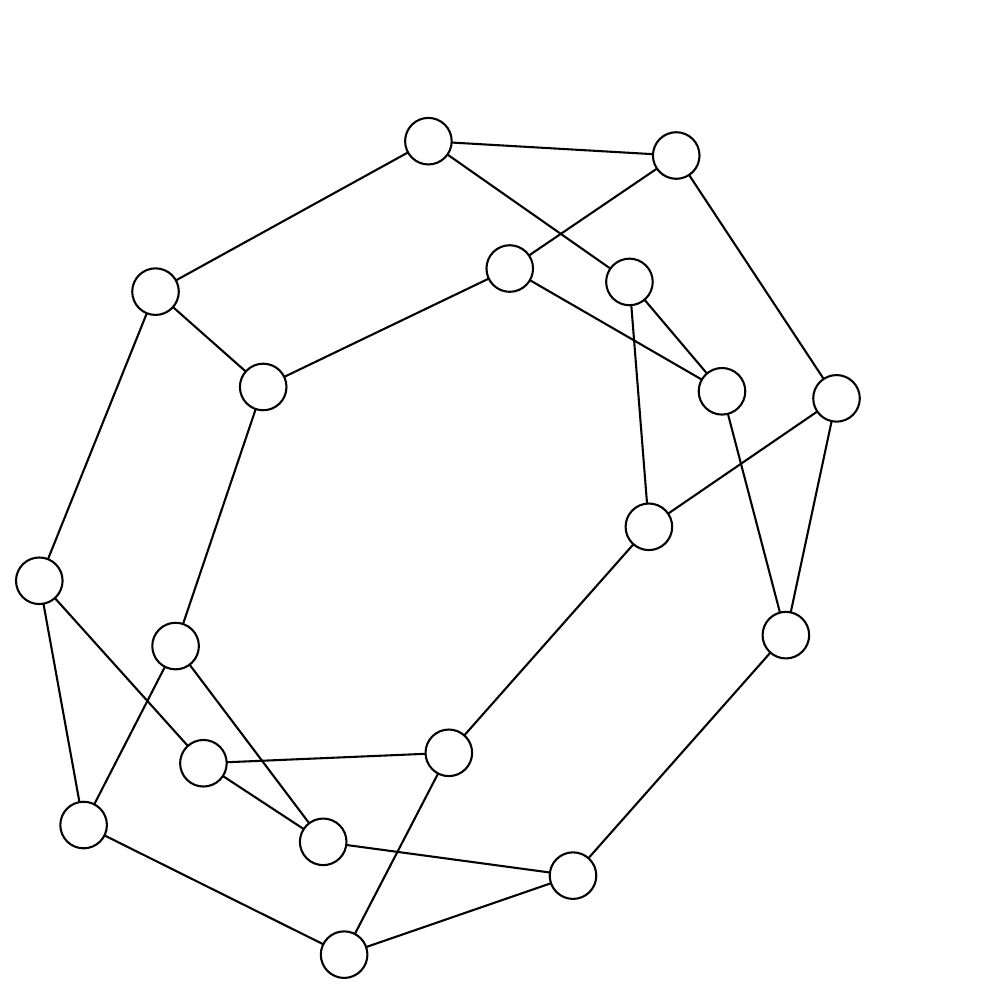}}
    \subfloat[]{\includegraphics[width=0.22\textwidth]{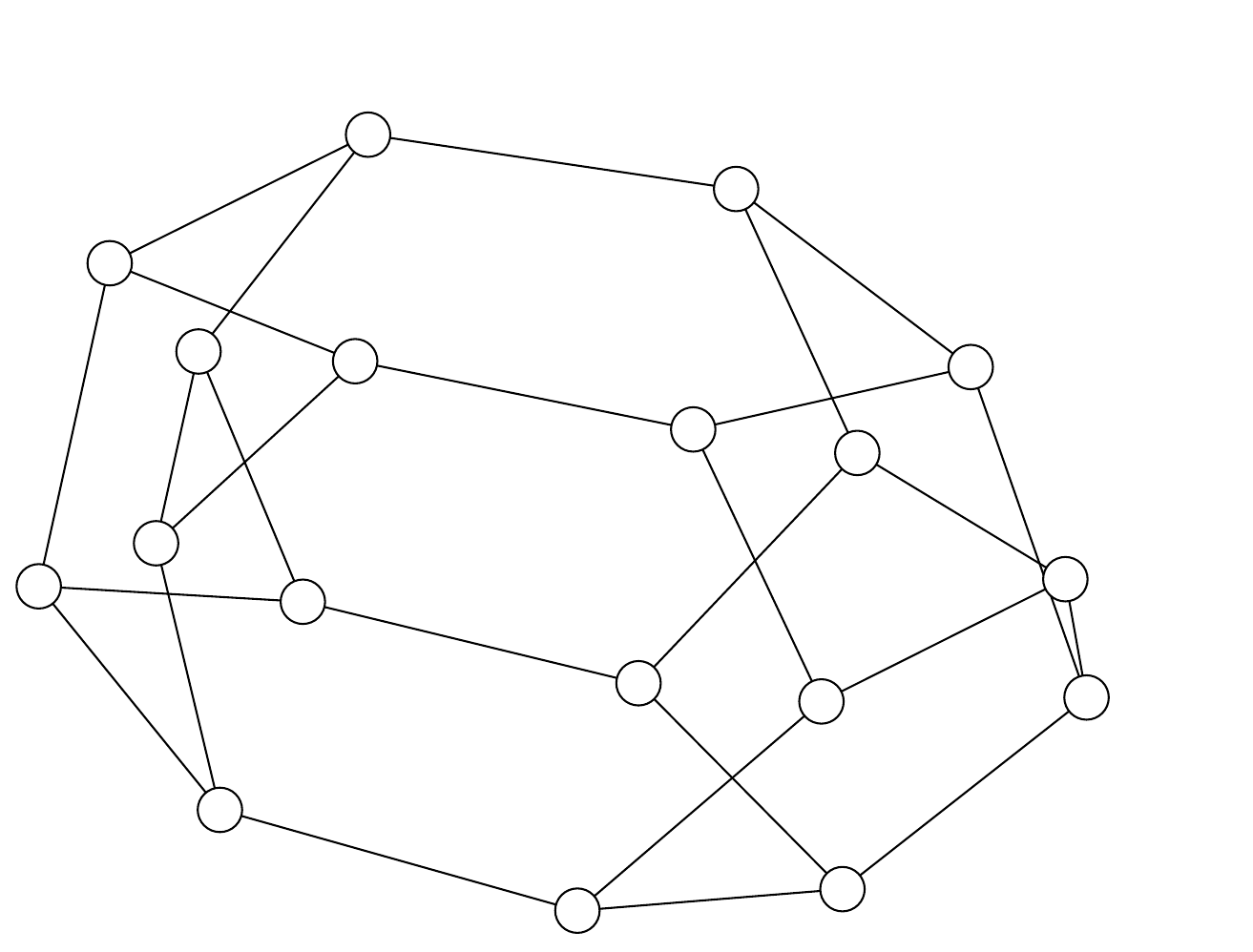}}
    \subfloat[]{\includegraphics[width=0.22\textwidth]{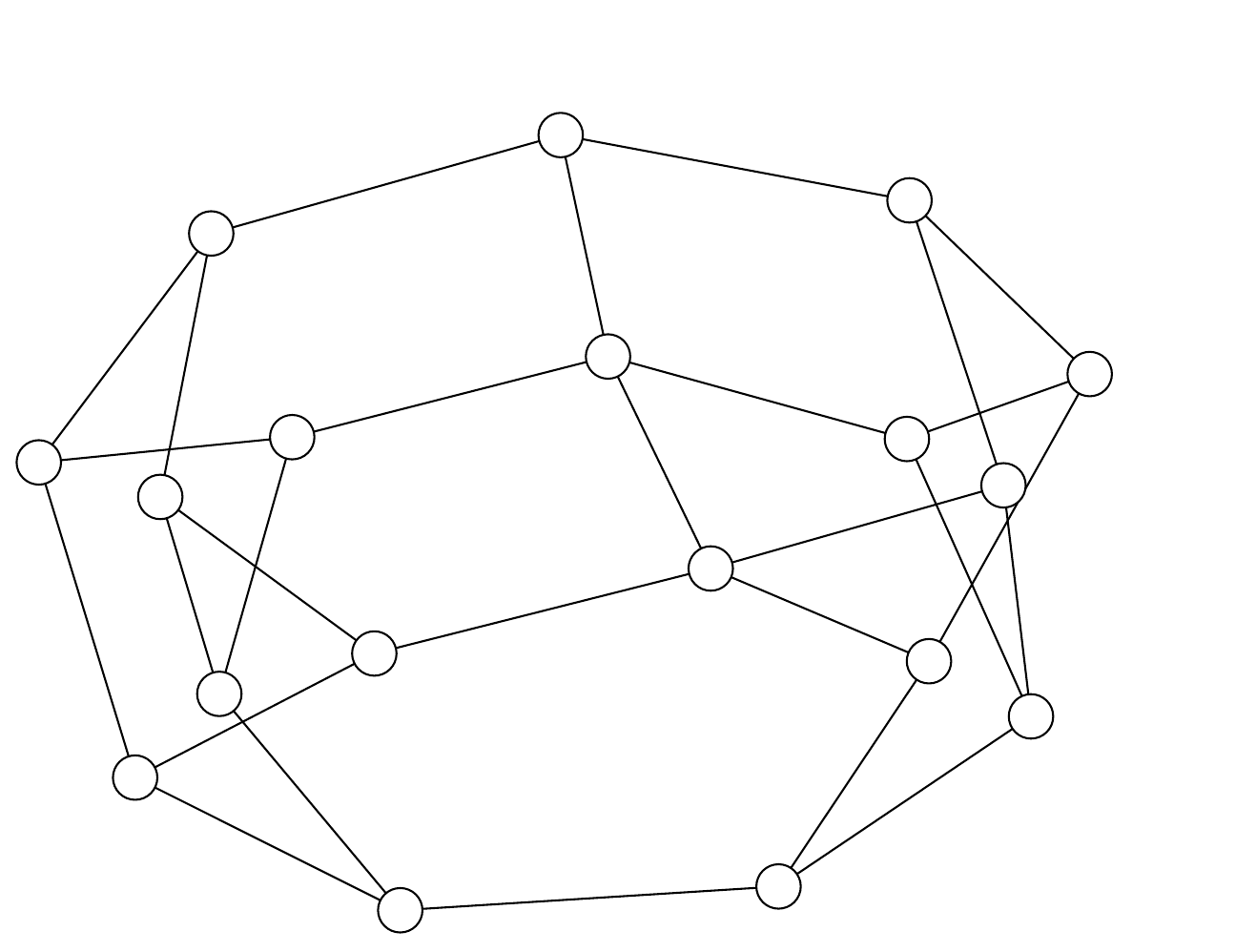}}
    \subfloat[]{\includegraphics[width=0.16\textwidth]{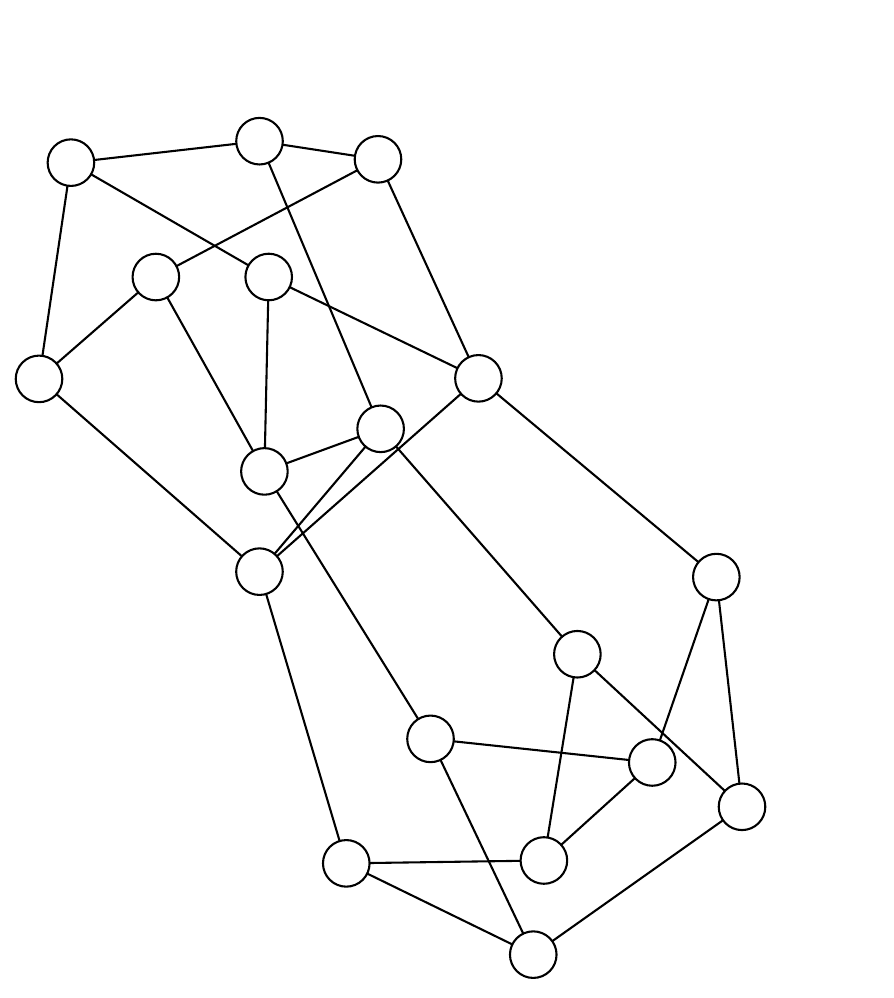}}\\
    \subfloat[]{\includegraphics[width=0.21\textwidth]{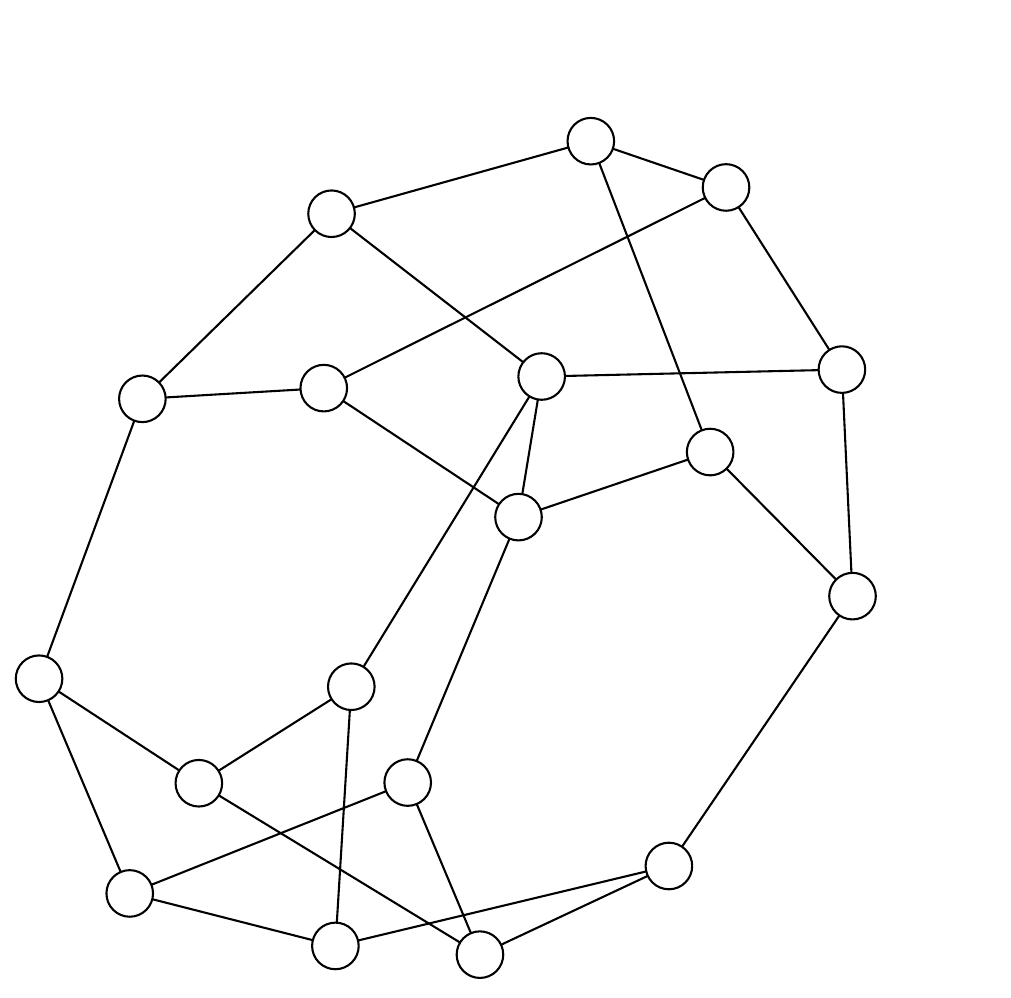}}
    \subfloat[]{\includegraphics[width=0.22\textwidth]{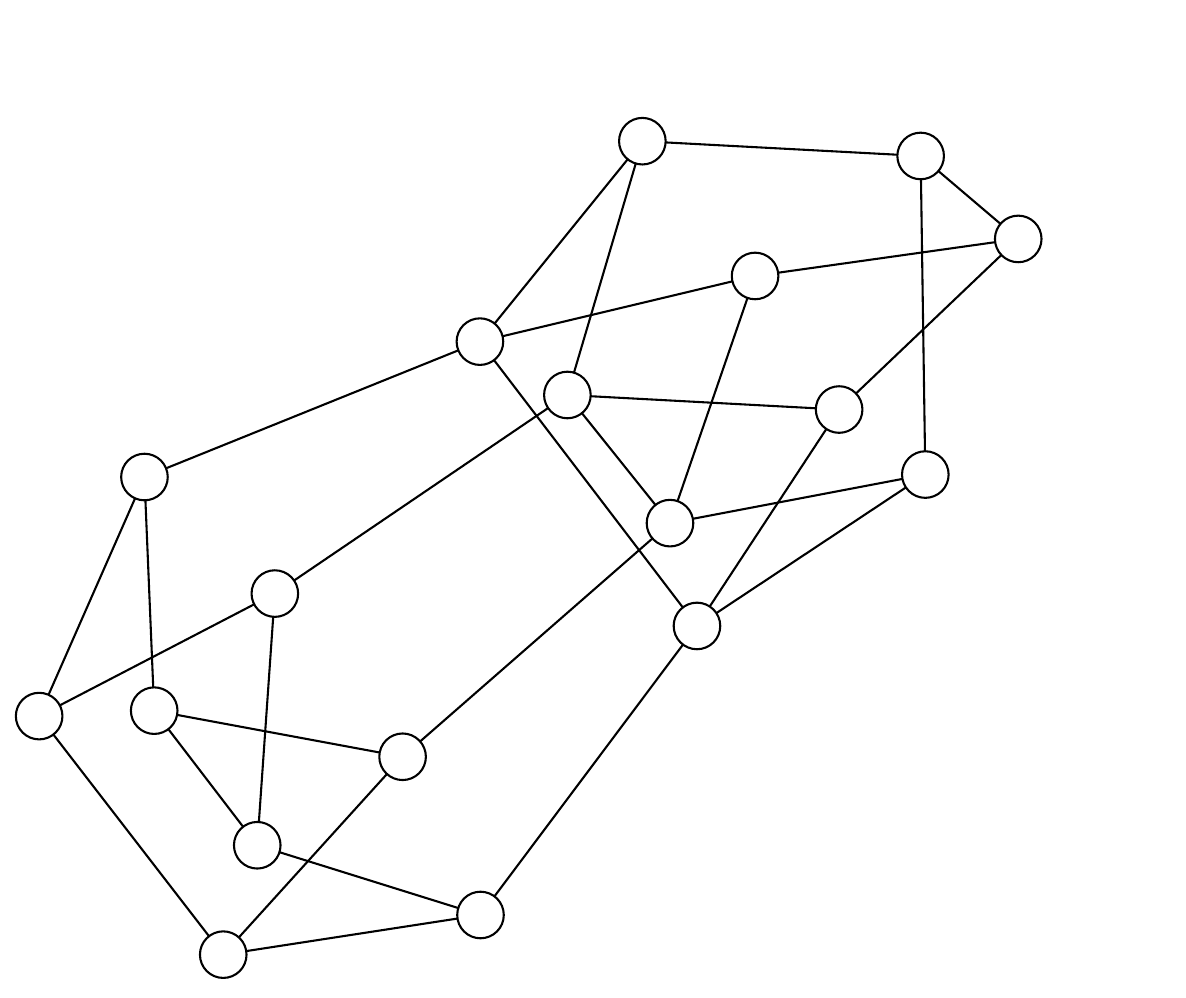}}
    \subfloat[]{\includegraphics[width=0.14\textwidth]{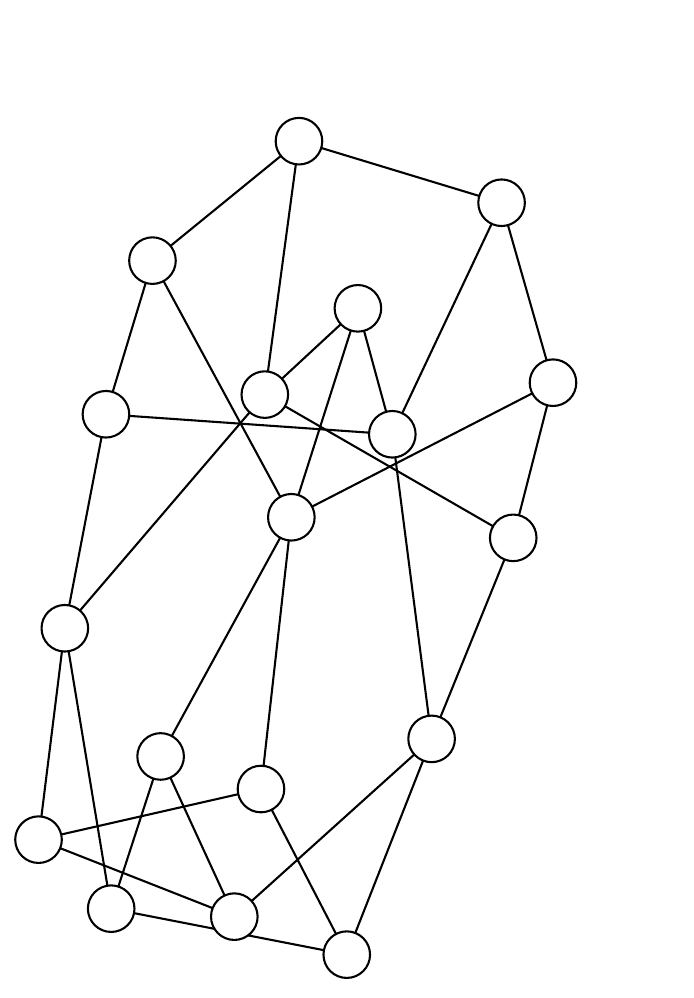}}
    \subfloat[]{\label{fig:hypo18-newgraph}\includegraphics[width=0.21\textwidth]{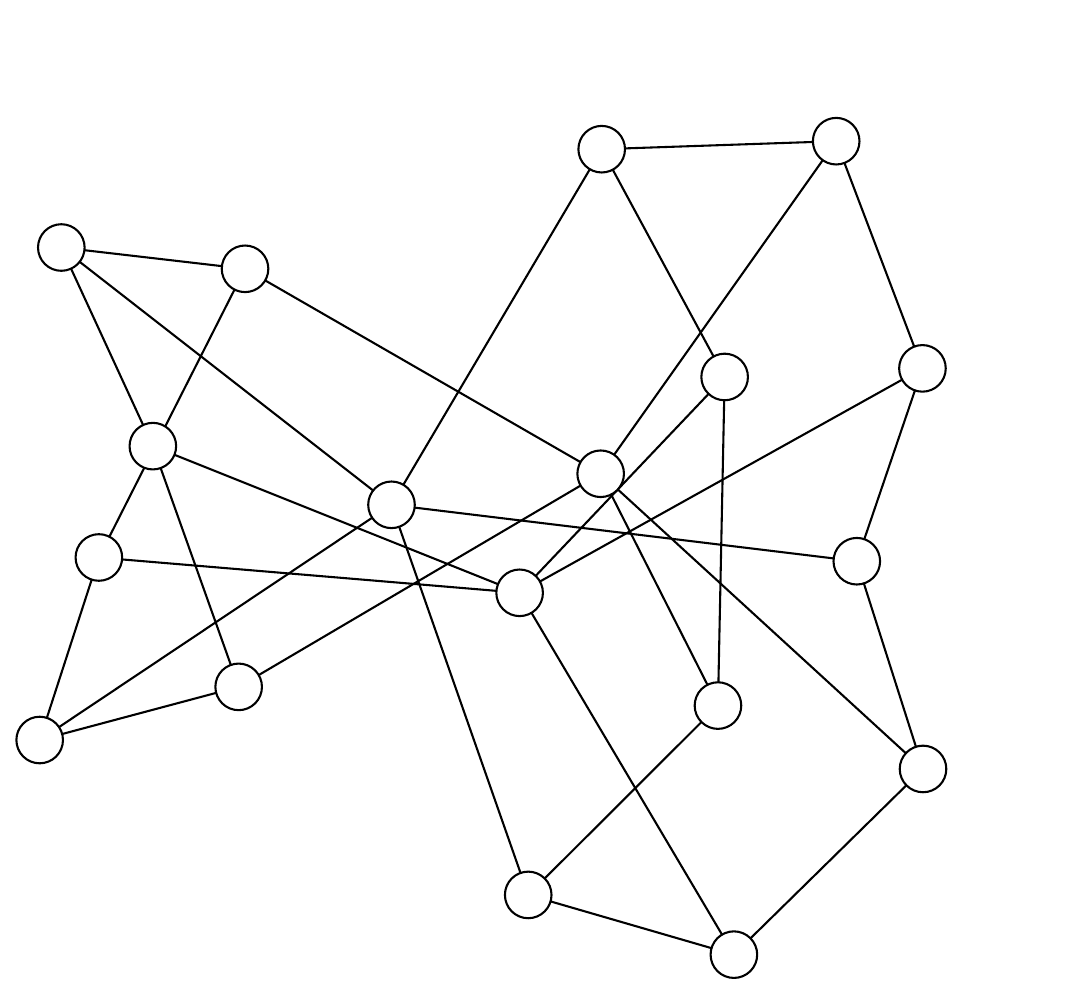}}
    \caption{All 14 hypohamiltonian graphs of order 18. Graph~(a) is the smallest hypohamiltonian graph of girth~3, while graphs (b)--(f) are the smallest hypohamiltonian graphs of girth~4.}
    \label{fig:hypo_18}
\end{figure}

Tables~\ref{table:times_hypo_g3}-\ref{table:times_hypo_g5} list the running times of the algorithm. The column \textit{``Max.\ nr.\ edges added''} denotes the maximum number of edges added by Algorithm~\ref{algo:construct} to a graph constructed by Algorithm~\ref{algo:init-algo} (i.e.\ the maximum number of recursive calls of \textit{Construct()}). 

The reported running times were obtained by executing our implementation of Algorithm~\ref{algo:init-algo} on an Intel Xeon CPU E5-2690 CPU at 2.90GHz. For the larger cases we did not include any running times in Tables~\ref{table:times_hypo_g3}-\ref{table:times_hypo_g5} since these were executed on a heterogeneous cluster and the parallelisation also caused a significant overhead. However in each case we went as far as computationally possible (most of the largest cases took between 1 and 10 CPU years).

Since the running times and number of intermediate graphs generated by the algorithm grows that fast, it seems very unlikely that these bounds can be improved in the near future using only faster computers.

Starting from girth at least~7, the bottleneck is the case where the generated graphs have maximum degree~3 (so here we are generating cubic hypohamiltonian graphs). (Also for girth 6, the cubic case forms a significant part of the total running time.) Algorithm~\ref{algo:init-algo} can also be used to generate only cubic hypohamiltonian graphs (and we also did this for correctness testing, see Section~\ref{sect:correctness_testing}). But here it is much more efficient to use a generator for cubic graphs with a given lower bound on the girth and testing if the generated graphs are hypohamiltonian as a filter. So for the generation of hypohamiltonian graphs with girth at least~6, we used Algorithm~\ref{algo:init-algo} only to construct hypohamiltonian graphs with maximum degree at least 4 and did the cubic case separately by using a generator for cubic graphs. More results on the cubic case can be found in Section~\ref{sect:cubic_case}.

Using Algorithm~\ref{algo:init-algo}, we have also determined the smallest hypohamiltonian graph of girth~6. It has 25~vertices and is shown in Figure~\ref{fig:hypo_25_g6}.

\begin{table}
\centering
\small
	\begin{tabular}{|c || c | r | r | c |}
		\hline
		Order & \# hypoham. & Time (s) & Increase & Max.\ nr.\ edges added\\
		\hline
16  &  4  &  9  &    &  15  \\
17  &  0  &  189  &  21.00  &  16  \\
18  &  14  &  18 339  &  97.03  &  18  \\
19  &  34 &    &    &    \\
		\hline
	\end{tabular}

\caption{Counts and generation times for hypohamiltonian graphs.}
\label{table:times_hypo_g3}

\end{table}

\begin{table}
\centering
\small
	\begin{tabular}{|c || c | r | r | c |}
		\hline
		Order & \# hypoham.\ $g \ge 4$ & Time (s) & Increase & Max.\ nr.\ edges added\\
		\hline
16  &  4  &  2  &    &  11  \\
17  &  0  &  19  &  9.50  &  12  \\
18  &  13  &  683  &  35.95  &  18  \\
19  &  34  &  10 816  &  15.84  &  19  \\
		\hline
	\end{tabular}

\caption{Counts and generation times for hypohamiltonian graphs with girth at least~4.}
\label{table:times_hypo_g4}

\end{table}

\begin{table}
\centering
\small
	\begin{tabular}{|c || c | r | r | c |}
		\hline
		Order & \# hypoham.\ $g \ge 5$ & Time (s) & Increase & Max.\ nr.\ edges added\\
		\hline
17  &  0  &  1  &    &  8  \\
18  &  8  &  9  &  9.00  &  9  \\
19  &  34  &  81  &  9.00  &  10  \\
20  &  4  &  1 125  &  13.89  &  11  \\
21  &  85  &  11 470  &  10.20  &  12  \\
22  &  420  &    &    &    \\
		\hline
	\end{tabular}

\caption{Counts and generation times for hypohamiltonian graphs with girth at least~5.}
\label{table:times_hypo_g5}

\end{table}

\begin{figure}[h!t]
	\centering
	\includegraphics[width=0.4\textwidth]{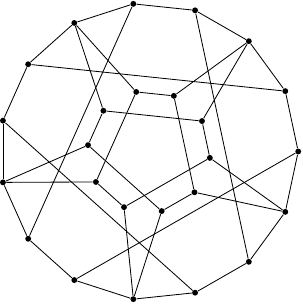}
	\caption{The smallest hypohamiltonian graph of girth~6. It has 25~vertices.}
	\label{fig:hypo_25_g6}
\end{figure}

\subsubsection{The cubic case}
\label{sect:cubic_case}

As already mentioned in the introduction, Aldred, McKay, and Wormald~\cite{AMW97} determined all cubic hypohamiltonian graphs up to 26~vertices and all cubic hypohamiltonian graphs with girth at least~5 and girth at least~6 on 28 and 30~vertices, respectively. In Table~\ref{table:counts_hypo_cubic} we extend these results. We used the program \textit{snarkhunter}~\cite{brinkmann_11,BG15} to generate all cubic graphs with girth at least $k$ for $4 \le k \le 7$, the program \textit{genreg}~\cite{meringer_99} for $k=8$ and the program of McKay et al.~\cite{mckay1998fast} for $k=9$. (Note that by Lemma~\ref{lem:triangle_obstr} cubic hypohamiltonian graphs must have girth at least~4.)

For girth at least $k$ for $k=7,8,9$ we obtained the following results:

\begin{theorem}\label{thm:cubic_high_girth}
By generating all cubic graphs with a given lower bound on the girth and testing them for hamiltonicity we obtained the following:
\begin{itemize}
\item The $28$-vertex Coxeter graph is the only non-hamiltonian cubic graph with girth~$7$ up to at least $42$~vertices.
\item The smallest non-hamiltonian cubic graph with girth~$8$ has at least $50$~vertices.
\item The smallest non-hamiltonian cubic graph with girth~$9$ has at least $66$~vertices.
\end{itemize}
\end{theorem}

Since hypohamiltonian graphs are non-hamiltonian, Theorem~\ref{thm:cubic_high_girth} also implies improved lower bounds for cubic hypohamiltonian graphs (see Table~\ref{table:bounds_hypoham_graphs}).

All hypohamiltonian graphs from Table~\ref{table:counts_hypo_cubic} can also be
downloaded from the \textit{House of Graphs}~\cite{hog} at
\url{http://hog.grinvin.org/Hypohamiltonian}~.

\begin{table}
\centering
%\small
\footnotesize
	\begin{tabular}{|c || r | r | r | r | r | r |}
		\hline
		\multirow{2}{*}{Order} & \multirow{2}{*}{$g \ge 4$} & Non-ham. & \multirow{2}{*}{Hypoham.} & Hypoham. & Hypoham. & Hypoham.\\		
		 &  & and $g \ge 4$ & & and $g \ge 5$ & and $g \ge 6$ & and $g \ge 7$\\				
%		 \# hypoham.\ $g \ge 5$ & Time (s) & Increase & Max.\ nr.\ edges added\\
		\hline
10  &  6  &  1  &  1  &  1  &  0  &  0 \\
12  &  22  &  0  &  0  &  0  &  0  &  0 \\
14  &  110  &  2  &  0  &  0  &  0  &  0 \\
16  &  792  &  8  &  0  &  0  &  0  &  0 \\
18  &  7 805  &  59  &  2  &  2  &  0  &  0 \\
20  &  97 546  &  425  &  1  &  1  &  0  &  0 \\
22  &  1 435 720  &  3 862  &  3  &  3  &  0  &  0 \\
24  &  23 780 814  &  41 293  &  1  &  0  &  0  &  0 \\
26  &  432 757 568  &  518 159  &  100  &  96  &  0  &  0 \\
28  &  8 542 471 494  &  7 398 734  &  52  &  34  &  2  &  1 \\
30  &  181 492 137 812  &  117 963 348  &  202  &  139  &  1  &  0 \\
32  &  4 127 077 143 862  &  2 069 516 990  &  304  &  28  &  0  &  0 \\
		\hline
	\end{tabular}

\caption{Counts of hypohamiltonian graphs among cubic graphs. $g$ stands for girth.}
\label{table:counts_hypo_cubic}

\end{table}

\subsection{Correctness testing}
\label{sect:correctness_testing}

To make sure that our implementation of Algorithm~\ref{algo:init-algo} did not contain any programming errors, we performed various correctness tests which we will describe in this section.

Previously, all hypohamiltonian graphs up to 17 vertices were known. We verified that our program yields exactly the same graphs. Aldred, McKay, and Wormald also produced a sample of 13 hypohamiltonian graphs with 18 vertices and a sample of 10 hypohamiltonian graphs with girth 5 and 22~vertices (see~\cite{mckay-site}). We verified that our program indeed also finds these graphs.

Our program can also be restricted to generate cubic hypohamiltonian graphs. To find cubic hypohamiltonian graphs of larger orders it is actually much more efficient to use a generator for cubic graphs and then test the generated graphs for hypohamiltonicity as a filter. However we used our program to generate cubic hypohamiltonian graphs as a correctness test. We used it to generate all cubic hypohamiltonian graphs up to 22~vertices---note that these graphs must have girth at least~4 due to Lemma~\ref{lem:triangle_obstr}---and all cubic hypohamiltonian graphs with girth at least 5 up to 24~vertices. These results were in complete agreement with the known results for cubic graphs from Section~\ref{sect:cubic_case}.

Our routines for testing hamiltonicity and hypohamiltonicity were already extensively used and tested before (for example they were used in~\cite{snark-paper} to search for hypohamiltonian snarks). We also used multiple independent programs to test hamiltonicity and hypohamiltonicity ---one of those programs was kindly provided to us by Gunnar Brinkmann---and in each case the results were in complete agreement.

Furthermore, our implementation of Algorithm~\ref{algo:init-algo} (i.e.\ the program \textit{GenHypohamiltonian}) is released as open source software and the code can be downloaded and inspected at~\cite{genhypo-site}.

\section{Generating planar hypohamiltonian graphs}
\label{section:planar_hypoham_graphs}

In the early seventies, Chv\'{a}tal~\cite{Ch73} raised the problem whether \emph{planar} hypohamiltonian graphs exist and Gr\"unbaum conjectured that they do not exist~\cite{Gr74}. In 1976,
Thomassen~\cite{Th76} constructed infinitely many such graphs, the smallest among them having order~105. Subsequently, smaller planar hypohamiltonian graphs were found by Hatzel~\cite{Ha79} (order~57), the second author and Zamfirescu~\cite{ZZ07} (order~48), Araya and Wiener~\cite{WA11} (order~42), and Jooyandeh, McKay, \"{O}sterg{\aa}rd, Pettersson,
and the second author~\cite{JMOPZ} (order~40). The latter three graphs are shown in Figure~\ref{fig:planar_48}. The 40-vertex example is the smallest known planar hypohamiltonian graph, together with other 24 graphs of the same order~\cite{JMOPZ}.

\begin{figure}[h!t]
	\centering
	\includegraphics[width=1.0\textwidth]{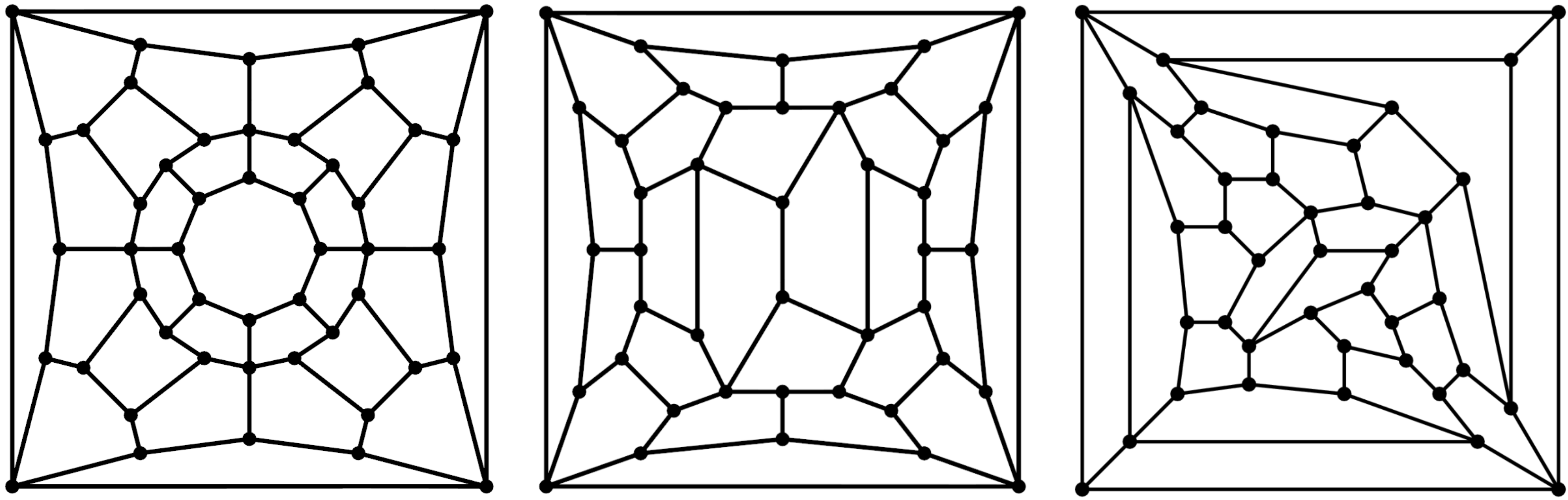}
	\caption{Planar hypohamiltonian graphs of order 48~\cite{ZZ07}, 42~\cite{WA11}, and 40~\cite{JMOPZ}, respectively.}
	\label{fig:planar_48}
\end{figure}

\subsection{The general case}

Jooyandeh, McKay, \"{O}sterg{\aa}rd, Pettersson, and the second author~\cite{JMOPZ} showed that the smallest planar hypohamiltonian graph of girth~5 has order~45, and that the graph with these properties is unique; see Figure~\ref{fig:planar_g5}.

\begin{figure}[h!t]
	\centering
	\includegraphics[width=0.4\textwidth]{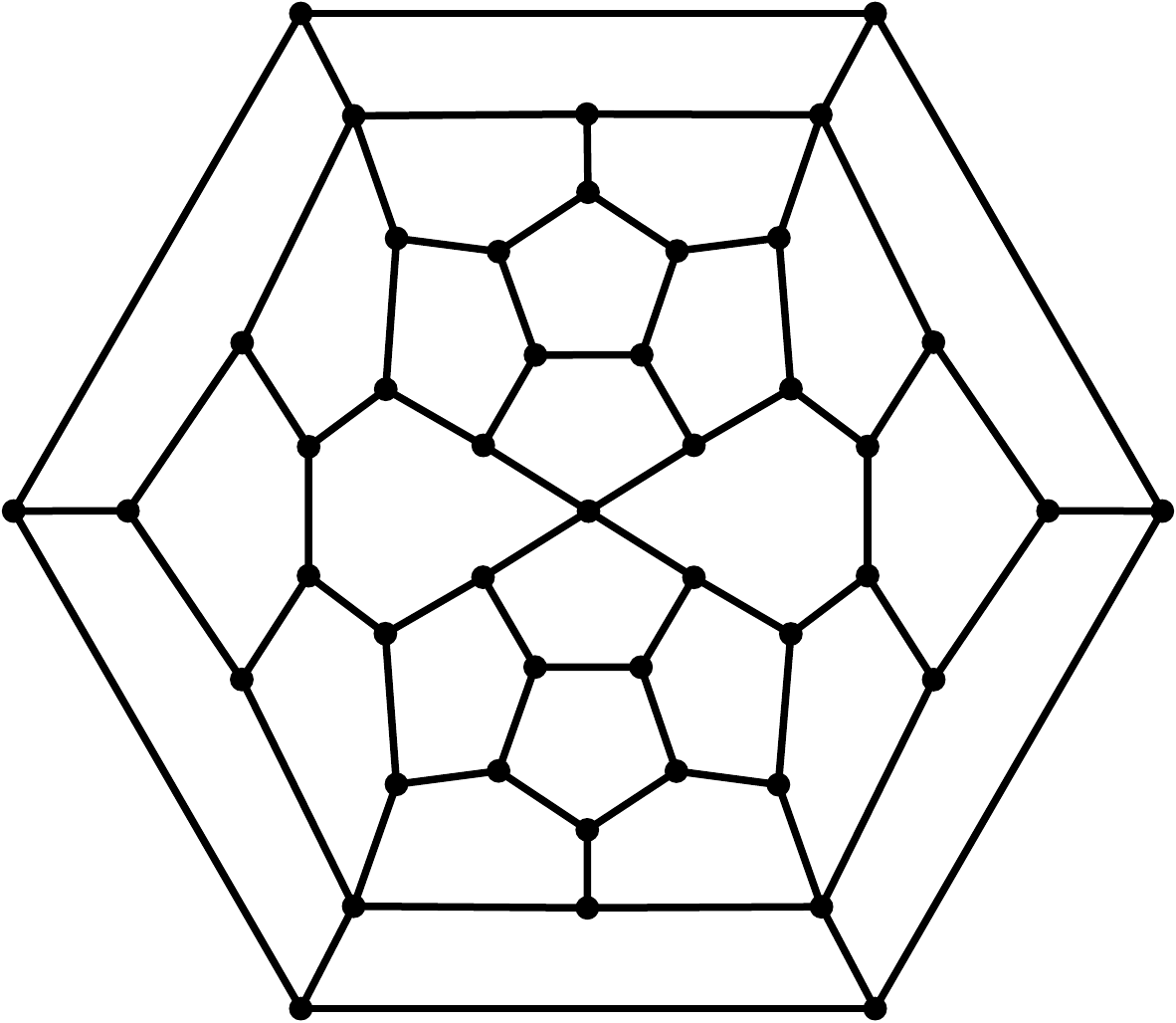}
	\caption{The unique planar hypohamiltonian graph of order $45$ and girth $5$. It was shown in~\cite{JMOPZ} that there is no smaller planar hypohamiltonian graph of girth~5.}
	\label{fig:planar_g5}
\end{figure}

Since planar hypohamiltonian graphs have girth at most~5 (due to Euler's formula), the smallest planar hypohamiltonian graph must have girth either~3 or~4. Thomassen~\cite{Th74-2} proved that, rather surprisingly, hypohamiltonian graphs of girth~3 exist. In~\cite{Th76}, Thomassen mentions how his approach from~\cite{Th74-2} can be applied to obtain a planar hypohamiltonian graph of girth~3. Using one of the aforementioned planar hypohamiltonian graphs of order~40 constructed in~\cite{JMOPZ}, one can obtain a planar hypohamiltonian graph of girth~3 and order~240. No smaller example is known.

Aldred, McKay, and Wormald~\cite{AMW97} showed that the smallest planar hypohamiltonian has order at least~18. Up until now, 18 was also the best lower bound for the order of the smallest planar hypohamiltonian graph. Jooyandeh, McKay, \"{O}sterg{\aa}rd, Pettersson, and the second author~\cite{JMOPZ} recently improved the upper bound from~42 to~40. In~\cite{JMOPZ}, the authors emphasise that no extensive computer search had been carried out to increase the lower bound for the smallest planar hypohamiltonian graph. This was one of the principal motivations of the present work.

Since the algorithm for generating all hypohamiltonian graphs presented in Section~\ref{section:hypoham_graphs} only adds edges and never removes any vertices or edges, all graphs obtained by the algorithm from a non-planar graph will remain non-planar. So in case we only want to generate planar hypohamiltonian graphs, we can prune the construction when a non-planar graph is constructed.

To this end we add a test for planarity on line~\ref{line:hamiltonian} of Algorithm~\ref{algo:construct}. We used Boyer and Myrvold's algorithm~\cite{boyer2004cutting} to test if a graph is planar.

\bigskip

\textbf{Additional properties of planar hypohamiltonian graphs}
\begin{itemize}
\item Using a theorem of Whitney~\cite{Wh31}, Thomassen showed~\cite{Th81} that a planar hypohamiltonian graph does not contain a maximal planar graph $G$, where $G \ne K_3$.
\item Let $G$ be a planar hypohamiltonian graph. Let $\kappa(G)$, $\lambda(G)$, and $\delta(G)$ denote the vertex-connectivity, minimum degree, and edge-connectivity of $G$, respectively. Then $\kappa(G) = \lambda(G) = \delta(G) = 3$ (for a proof, see~\cite{JMOPZ}).
\end{itemize}

We also present a result from \cite{JMOPZ} which restricts the family of polyhedra in which the smallest planar hypohamiltonian graph must reside. For further details, see~\cite{JMOPZ}. In that article, the operation 4-\emph{face deflater} ${\cal FD}_4$ is defined which squeezes a 4-face of a plane graph into a path of length~2. The inverse of this operation is called a 2-\emph{path inflater} ${\cal PI}_2$, which expands a path of length~2 into a 4-face. Let ${\cal D}_5 (f)$ be the set of all plane graphs with $f$~faces and minimum degree at least~5. Let $G^\star$ denote the dual of a planar graph $G$, and put $${\cal M}_f^4(n) = \begin{cases}
                            \{ G^\star : G \in {\cal D}_5(n) \}  & f = 0\\
                            \bigcup_{G \in M_{f-1}^4(n-1)} {\cal PI}_2(G) & f > 0
                        \end{cases} \quad {\rm and} \quad {\cal M}_f^4 = \bigcup_n {\cal M}_f^4(n).$$

\begin{theorem}[Jooyandeh et al.~\cite{JMOPZ}]
Let $G$ be the smallest planar hypohamiltonian graph. Then $G \notin {\cal M}_f^4$.
\end{theorem}

We extended our algorithm from Section~\ref{section:hypoham_graphs} to generate planar hypohamiltonian graphs and obtained the following results with it.

\begin{theorem}\label{thm:bound_planar}
The smallest planar hypohamiltonian graph has at least $23$~vertices.
\end{theorem}

\begin{theorem}\label{thm:bound_planar_g4}
The smallest planar hypohamiltonian graph with girth at least~$4$ has at least $27$~vertices.
\end{theorem}

When we combine this with the known upper bounds, we get the following corollary.

\begin{corollary}
Let $h$ ($h_g$) denote the order of the smallest planar hypohamiltonian graph (of girth $g$). We have $$23 \le h \le 40, \quad 23 \le h_3 \le 240, \quad 27 \le h_4 \le 40, \quad {\rm {\it and}} \quad h_5 = 45.$$
\end{corollary}

The running times of our implementation of this algorithm restricted to planar graphs is given in Tables~\ref{table:times_planar_hypo} and~\ref{table:times_planar_hypo_g4}. For the larger cases we did not include any running times since these were executed on a heterogeneous cluster and the parallelisation also caused a non-negligible overhead. The column \textit{``Max.\ nr.\ edges added''} denotes the maximum number of edges added by Algorithm~\ref{algo:construct} to a graph constructed by Algorithm~\ref{algo:init-algo}.

\begin{table}
\centering
\small
	\begin{tabular}{|c || c | r | r | c |}
		\hline
		Order & \# hypoham. & Time (s) & Increase & Max.\ nr.\ edges added\\
		\hline
16  &  0  &  4  &    &   9\\
17  &  0  &  35  &  8.75  &   11\\
18  &  0  &  235  &  6.71  &  14 \\
19  &  0  &  1 245  &  5.30  &  16 \\
20  &  0  &  13 517  &  10.86  &  17 \\
21  &  0  &  109 294  &  8.09  &  19 \\
22  &  0  &    &    &   \\
		\hline
	\end{tabular}

\caption{Counts and generation times for planar hypohamiltonian graphs.}
\label{table:times_planar_hypo}

\end{table}

\begin{table}
\centering
\small
	\begin{tabular}{|c || c | r | r | c |}
		\hline
		Order &  \# hypoham.\ $g \geq 4$ & Time (s) & Increase & Max.\ nr.\ edges added\\
		\hline
16  &  0  &  2  &    &   6\\
17  &  0  &  11  &  5.50  &  7 \\
18  &  0  &  35  &  3.18  &  8 \\
19  &  0  &  231  &  6.60  &  10 \\
20  &  0  &  1 649  &  7.14  &  10 \\
21  &  0  &  9 545  &  5.79  &  12 \\
22  &  0  &  53 253  &  5.58  &  12 \\
23  &  0  &    &    &   \\
24  &  0  &    &    &   \\
		\hline
	\end{tabular}

\caption{Counts and generation times for planar hypohamiltonian graphs with girth at least 4.}
\label{table:times_planar_hypo_g4}

\end{table}

\subsection{The cubic case}

\label{section:planar_cubic}

Chv\'{a}tal~\cite{Ch73} asked in 1973 whether cubic planar hypohamiltonian graphs exist. His question was settled in 1981 by Thomassen~\cite{Th81}, who constructed such graphs of order $94+4k$ for every $k \ge 0$. However, the following two questions raised in~\cite{HS93} remained open: (i)~Are there smaller cubic planar hypohamiltonian graphs? (ii)~Does there exist a positive integer $n_0$ such that for every even~$n \ge n_0$ there exists a cubic planar hypohamiltonian graph of order~$n$? Araya and Wiener answered both of these questions affirmatively in~\cite{AW11}. Concerning~(i), they showed that there exists a cubic planar hypohamiltonian graph of order~70. No smaller such graph is known. Regarding~(ii), Araya and Wiener~\cite{AW11} showed that there exists a cubic planar hypohamiltonian graph of order $n$ for every even $n \ge 86$. The second author~\cite{Za15} improved this result by showing that such graphs exist for every even $n \ge 74$.

Until recently, all known cubic planar hypohamiltonian graphs had girth~4. (Recall that by Lemma~\ref{lem:triangle_obstr} cubic hypohamiltonian graphs must have girth at least 4). Due to a recent result of McKay~\cite{Mc}, we now know that cubic planar hyohamiltonian graphs of girth 5 exist, and that the smallest ones have order~76. So the smallest cubic planar hypohamiltonian must have girth exactly~4.

From the results of Aldred, Bau, Holton, and McKay~\cite{ABHM00} it follows that there is no cubic planar hypohamiltonian graph on $42$ or fewer vertices. (Completing the work of many researchers, Holton and McKay~\cite{HM} showed that the order of the smallest non-hamiltonian cubic planar $3$-connected graph is~38; one of the graphs realising this minimum is the famous Lederberg-Bos\'{a}k-Barnette graph). Moreover, all 42-vertex graphs presented in~\cite{ABHM00} have exactly one face whose size is not congruent to 2~modulo~3, and it was already observed by Thomassen~\cite{Th74-1} that such a graph cannot be hypohamiltonian. Summarising, prior to this work we knew that the smallest planar hypohamiltonian graph has girth~4 and order at least~44 and at most~70. %We now improve the lower bound.

\bigskip

\textbf{Additional properties of cubic planar hypohamiltonian graphs}

\bigskip

We now also mention obstructions specifically for cubic planar hypohamiltonian graphs. For the first obstruction below, we call a face~$F$ a \emph{$k$-face} if ${\rm size}(F) = k {\rm \ mod \ } 3$. Let $G$ be a cubic planar hypohamiltonian graph.

\begin{itemize}
\item Araya and Wiener~\cite{AW11} extended a remark of Thomassen~\cite{Th74-1} and showed that (i) $G$ contains at least three non-$2$-faces, (ii) if $G$ has exactly three non-$2$-faces, then these three non-$2$-faces do not have a common vertex, and (iii) two 1-faces or a 1-face and a 0-face cannot be adjacent.

\item Kardo{\v{s}}~\cite{Ka} has recently proven Barnette's conjecture which states that every cubic planar 3-connected graph in which each face has size at most~6 is hamiltonian. Thus, $G$ must contain a face of size at least~7.
\end{itemize}

By using the program \textit{plantri}~\cite{brinkmann_07} we generated all cubic planar cyclically 4-connected graphs with girth~4 up to 52~vertices and tested them for hypohamiltonicity. (Note that prior to our result, the best lower bound for the order of the smallest cubic planar hypohamiltonian graph was~44, see~\cite{AW11}). No hypohamiltonian graphs were found, so we have in summary the following.

\begin{theorem}
The smallest cubic planar hypohamiltonian graph has girth~$4$, at least $54$ and at most $70$~vertices.
\end{theorem}

As mentioned earlier, McKay~\cite{Mc} recently showed that there exist no cubic planar hypohamiltonian graphs of girth~5 with less than 76~vertices, and exactly three such graphs of order~76. All three graphs have trivial automorphism group. In that paper the natural question is raised whether infinitely many such graphs exist. Using the program \textit{plantri}~\cite{brinkmann_07} we generated all cubic planar cyclically 4-connected graphs with girth 5 with 78~vertices and tested them for hypohamiltonicity. This yielded exactly one such graph. Although we are not able to settle McKay's question, in the following theorem we make a first step.

\begin{theorem}
\label{thm:cubic_planar_g5_nontriv}
There is exactly one cubic planar hypohamiltonian graph of order~$78$ and girth~$5$. This graph is shown in Figure~\ref{fig:planar_cubic_g5}. It is the smallest cubic planar hypohamiltonian graph of girth~$5$ with a non-trivial automorphism group and has $D_{3h}$ symmetry (as an abstract group, this is the dihedral group of order~$12$).
\end{theorem}

The graph from Theorem~\ref{thm:cubic_planar_g5_nontriv} can also be downloaded and inspected at the database of interesting graphs from the \textit{House of Graphs}~\cite{hog} by searching for the keywords ``hypohamiltonian * D3h".

\begin{figure}[h!t]
    \centering
    \subfloat[]{\label{fig:planar_cubic_g5a}\includegraphics[width=0.45\textwidth]{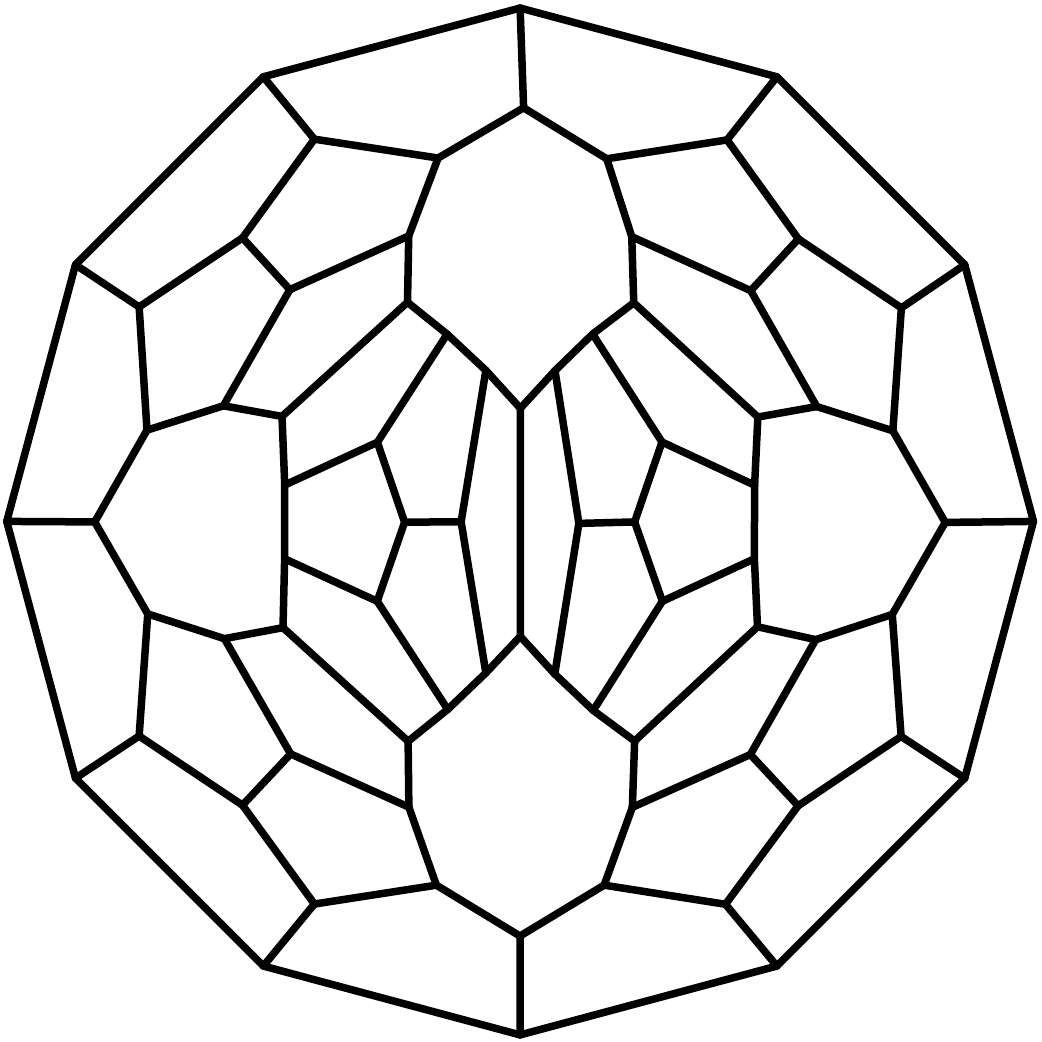}} \qquad
    \subfloat[]{\label{fig:planar_cubic_g5b}\includegraphics[width=0.45\textwidth]{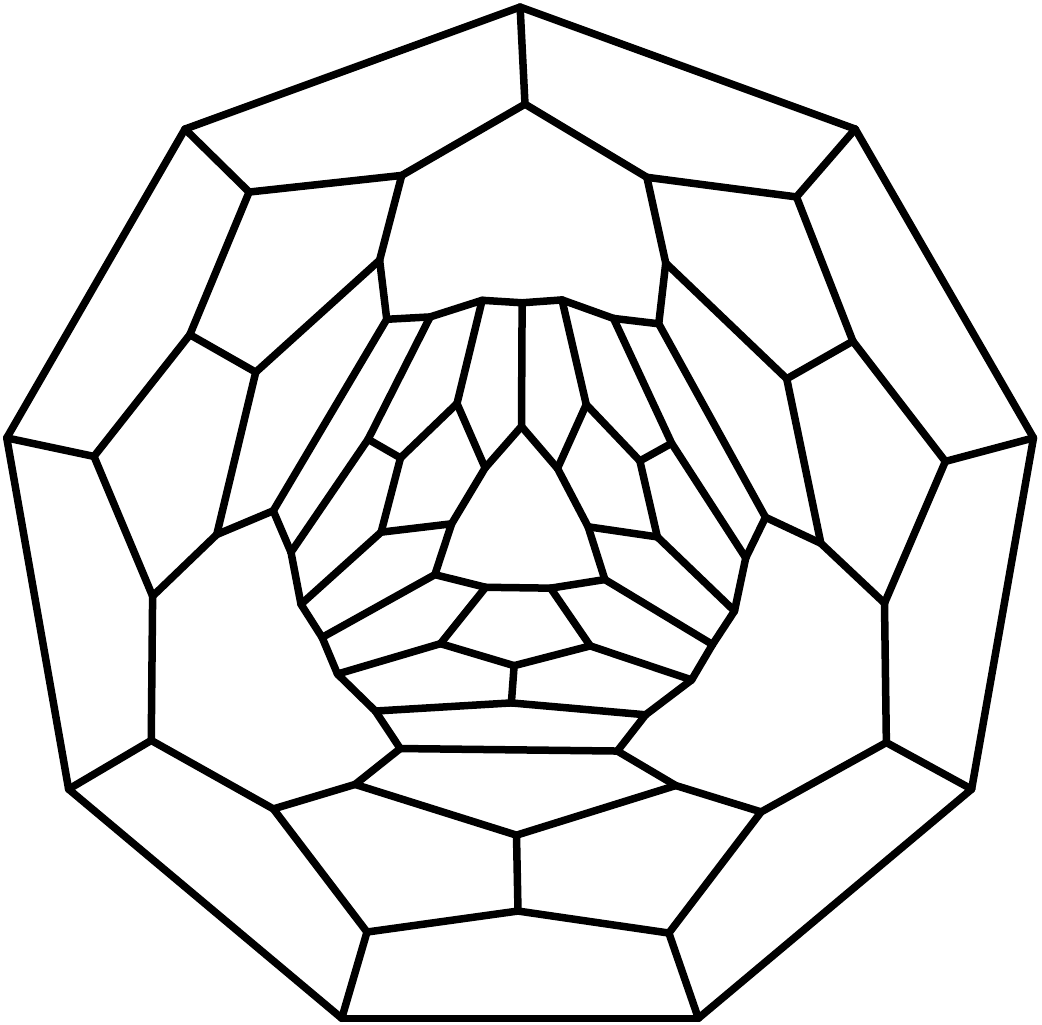}}\\
    \caption{The smallest cubic planar hypohamiltonian graph of girth~$5$ with a non-trivial automorphism group. It has 78 vertices and $D_{3h}$ symmetry. Both Figure~\ref{fig:planar_cubic_g5a} and Figure~\ref{fig:planar_cubic_g5b} show different symmetries of the same graph.}
    \label{fig:planar_cubic_g5}
\end{figure}

\section{Outlook}

We would like to conclude with comments and open questions which might be worth pursuing as future work.

\begin{enumerate}
\item We have seen that the order of the smallest planar hypohamiltonian graph must lie between 23 and 40. Let us read ``being planar'' as ``having crossing number~0''. It is not difficult to show that the Petersen graph is the smallest hypohamiltonian graph with crossing number~2, see e.g.~\cite{Za12}. The second author showed in~\cite{Za12} that there exists a hypohamiltonian graph with crossing number~1 and order~46. Recently, Wiener~\cite{Wi16} constructed a hypohamiltonian graph with crossing number~1 and order~36. This is the smallest example up to date---so we ask here: what is the order of the smallest hypohamiltonian graph with crossing number~1?

\item In the deep and technical paper~\cite{Sa87}, Sanders defines a graph $G$ to be \emph{almost hamiltonian} if every subset of $|V(G)| - 1$ vertices is contained in a cycle. Every hypocyclic (and thus every hypohamiltonian) graph is almost hamiltonian, but the converse is not necessarily true: take a hamiltonian graph $G$ in which there exists a vertex $v$ such that $G - v$ is not hamiltonian. Sanders characterises almost hamiltonian graphs in terms of circuit injections and binary matroids (for the definitions, see~\cite{Sa87}). Possibly an algorithmic implementation of Sanders' characterisation is worth pursuing.

\item Ad finem, we discuss the order of the smallest planar hypohamiltonian graph. In this article, we have increased the lower bound from 18 to 23, but there is still a considerable gap to 40, the best available upper bound~\cite{JMOPZ}. As mentioned in~\cite{JMOPZ}, it would be somewhat surprising if every extremal graph would have trivial automorphism group---note that the smallest planar hypohamiltonian graphs we know of, the 40-vertex graphs from~\cite{JMOPZ}, all have only identity as automorphism. An exhaustive search for graphs with prescribed automorphisms might lead to smaller planar hypohamiltonian graphs. 
\end{enumerate}

\section*{Acknowledgements}
Most computations for this work were carried out using the Stevin Supercomputer Infrastructure at Ghent University. We also would like to thank Gunnar Brinkmann for providing us with an independent program for testing hypohamiltonicity.

%% Bibliography

\bibliographystyle{plain}
\bibliography{references}

\begin{thebibliography}{10}

\bibitem{ABHM00}
R.E.L. Aldred, S.~Bau, D.A. Holton, and B.D. McKay.
\newblock Nonhamiltonian 3-connected cubic planar graphs.
\newblock {\em SIAM Journal on Discrete Mathematics}, 13(1):25--32, 2000.

\bibitem{AMW97}
R.E.L. Aldred, B.D. McKay, and N.C. Wormald.
\newblock Small hypohamiltonian graphs.
\newblock {\em Journal of Combinatorial Mathematics and Combinatorial
  Computing}, 23:143--152, 1997.

\bibitem{AW11}
M.~Araya and G.~Wiener.
\newblock On cubic planar hypohamiltonian and hypotraceable graphs.
\newblock {\em The Electronic Journal of Combinatorics}, 18(1):1--11, 2011.

\bibitem{boyer2004cutting}
J.M. Boyer and W.J. Myrvold.
\newblock {On the Cutting Edge: Simplified $O(n)$ Planarity by Edge Addition}.
\newblock {\em Journal of Graph Algorithms and Applications}, 8(2):241--273,
  2004.

\bibitem{hog}
G.~Brinkmann, K.~Coolsaet, J.~Goedgebeur, and H.~M{\'e}lot.
\newblock {House of Graphs: a database of interesting graphs}.
\newblock {\em Discrete Applied Mathematics}, 161(1-2):311--314, 2013.
\newblock Available at \url{http://hog.grinvin.org/}.

\bibitem{BG15}
G.~Brinkmann and J.~Goedgebeur.
\newblock Generation of cubic graphs and snarks with large girth.
\newblock {\em Journal of Graph Theory}, 86(2):255--272, 2017.

\bibitem{snark-paper}
G.~Brinkmann, J.~Goedgebeur, J.~H{\"a}gglund, and K.~Markstr{\"o}m.
\newblock Generation and properties of snarks.
\newblock {\em Journal of Combinatorial Theory, Series B}, 103(4):468--488,
  2013.

\bibitem{brinkmann_11}
G.~Brinkmann, J.~Goedgebeur, and B.D. McKay.
\newblock Generation of cubic graphs.
\newblock {\em Discrete Mathematics and Theoretical Computer Science},
  13(2):69--80, 2011.

\bibitem{brinkmann_07}
G.~Brinkmann and B.D. McKay.
\newblock Fast generation of planar graphs.
\newblock {\em MATCH Commun. Math. Comput. Chem.}, 58(2):323--357, 2007.

\bibitem{CKL70}
G.~Chartrand, S.F. Kapoor, and D.R. Lick.
\newblock $n$-hamiltonian graphs.
\newblock {\em Journal of Combinatorial Theory}, 9(3):308--312, 1970.

\bibitem{Ch73}
V.~Chv\'atal.
\newblock Flip-flops in hypohamiltonian graphs.
\newblock {\em Canadian Mathematical Bulletin}, 16(1):33--41, 1973.

\bibitem{CS77}
J.B. Collier and E.F. Schmeichel.
\newblock New flip-flop constructions for hypohamiltonian graphs.
\newblock {\em Discrete Mathematics}, 18(3):265--271, 1977.

\bibitem{CS78}
J.B. Collier and E.F. Schmeichel.
\newblock Systematic searches for hypohamiltonian graphs.
\newblock {\em Networks}, 8(3):193--200, 1978.

\bibitem{Di10}
R.~Diestel.
\newblock Graph {T}heory.
\newblock {\em Gradudate Texts in Mathematics}, 173, 2010.

\bibitem{GZ}
J.~Goedgebeur and C.~T. Zamfirescu.
\newblock Improved bounds for hypohamiltonian graphs.
\newblock {\em Ars Mathematica Contemporanea}, 13(2):235--257, 2017.

\bibitem{genhypo-site}
J.~Goedgebeur and C.T. Zamfirescu.
\newblock Homepage of {GenHypohamiltonian}:
  \url{http://caagt.ugent.be/hypoham/}.

\bibitem{GZ2}
J.~Goedgebeur and C.T. Zamfirescu.
\newblock On almost hypohamiltonian graphs.
\newblock {\em arXiv preprint arXiv:1606.06577}, 2016.

\bibitem{Gr74}
B.~Gr{\"u}nbaum.
\newblock Vertices missed by longest paths or circuits.
\newblock {\em Journal of Combinatorial Theory, Series A}, 17(1):31--38, 1974.

\bibitem{Ha79}
W.~Hatzel.
\newblock {Ein planarer hypohamiltonscher Graph mit 57 Knoten}.
\newblock {\em Mathematische Annalen}, 243(3):213--216, 1979.

\bibitem{HDV67}
J.C. Herz, J.J. Duby, and F.~Vigu{\'e}.
\newblock Recherche syst{\'e}matique des graphes hypohamiltoniens.
\newblock In {\em Theory of Graphs: International Symposium, Rome}, pages
  153--159, 1967.

\bibitem{HM}
D.A. Holton and B.D. McKay.
\newblock The smallest non-hamiltonian 3-connected cubic planar graphs have 38
  vertices.
\newblock {\em Journal of Combinatorial Theory, Series B}, 45(3):305--319,
  1988.

\bibitem{HS93}
D.A. Holton and J.~Sheehan.
\newblock {\em The {P}etersen Graph, Chapter 7: Hypohamiltonian graphs}.
\newblock Cambridge University Press, New York, 1993.

\bibitem{JMOPZ}
M.~Jooyandeh, B.D. McKay, P.R.J. {\"O}sterg{\aa}rd, V.H. Pettersson, and C.T.
  Zamfirescu.
\newblock Planar hypohamiltonian graphs on 40 vertices.
\newblock {\em Journal of Graph Theory}, 84(2):121--133, 2017.

\bibitem{Ka}
F.~Kardo\v{s}.
\newblock A computer-assisted proof of {B}arnette's conjecture: Not only
  fullerene graphs are hamiltonian.
\newblock {\em arXiv preprint arXiv:1409.2440}, 2014.

\bibitem{MS11}
E.~M{\'a}{\v{c}}ajov{\'a} and M.~{\v{S}}koviera.
\newblock Infinitely many hypohamiltonian cubic graphs of girth 7.
\newblock {\em Graphs and Combinatorics}, 27(2):231--241, 2011.

\bibitem{nauty-website}
B.D. McKay.
\newblock {nauty User's Guide (Version 2.5)}.
\newblock Technical Report TR-CS-90-02, Department of Computer Science,
  Australian National University. The latest version of the software is
  available at \url{http://cs.anu.edu.au/~bdm/nauty}.

\bibitem{mckay-site}
B.D. McKay.
\newblock Page with hypohamiltonian graphs:
  \url{http://users.cecs.anu.edu.au/~bdm/data/graphs.html}.

\bibitem{mckay_98}
B.D. McKay.
\newblock Isomorph-free exhaustive generation.
\newblock {\em Journal of Algorithms}, 26(2):306--324, 1998.

\bibitem{Mc}
B.D. McKay.
\newblock Hypohamiltonian planar cubic graphs with girth 5.
\newblock {\em Journal of Graph Theory}, 85(1):7--11, 2017.

\bibitem{mckay1998fast}
B.D. McKay, W.~Myrvold, and J.~Nadon.
\newblock Fast backtracking principles applied to find new cages.
\newblock {\em 9th Annual ACM-SIAM Symposium on Discrete Algorithms}, pages
  188--191, 1998.

\bibitem{mckay_14}
B.D. McKay and A.~Piperno.
\newblock Practical graph isomorphism, {II}.
\newblock {\em Journal of Symbolic Computation}, 60:94--112, 2014.

\bibitem{meringer_99}
M.~Meringer.
\newblock Fast generation of regular graphs and construction of cages.
\newblock {\em Journal of Graph Theory}, 30(2):137--146, 1999.

\bibitem{Sa87}
J.H. Sanders.
\newblock Circuit preserving edge maps {II}.
\newblock {\em Journal of Combinatorial Theory, Series B}, 42(2):146--155,
  1987.

\bibitem{So63}
R.~Sousselier.
\newblock Probl\`{e}me no. 29: Le cercle des irascibles.
\newblock {\em Revue fran{\c{c}}aise de recherche op{\'e}rationelle},
  7:405--406, 1963.

\bibitem{Th74-1}
C.~Thomassen.
\newblock Hypohamiltonian and hypotraceable graphs.
\newblock {\em Discrete Mathematics}, 9(1):91--96, 1974.

\bibitem{Th74-2}
C.~Thomassen.
\newblock On hypohamiltonian graphs.
\newblock {\em Discrete Mathematics}, 10(2):383--390, 1974.

\bibitem{Th76}
C.~Thomassen.
\newblock Planar and infinite hypohamiltonian and hypotraceable graphs.
\newblock {\em Discrete Mathematics}, 14(4):377--389, 1976.

\bibitem{Th78}
C.~Thomassen.
\newblock Hypohamiltonian graphs and digraphs.
\newblock In {\em Theory and Applications of Graphs}, pages 557--571. Springer,
  1978.

\bibitem{Th81}
C.~Thomassen.
\newblock Planar cubic hypohamiltonian and hypotraceable graphs.
\newblock {\em Journal of Combinatorial Theory, Series B}, 30(1):36--44, 1981.

\bibitem{Wh31}
H.~Whitney.
\newblock A theorem on graphs.
\newblock {\em Annals of Mathematics}, pages 378--390, 1931.

\bibitem{Wi16}
G.~Wiener.
\newblock New constructions of hypohamiltonian and hypotraceable graphs.
\newblock {\em Journal of Graph Theory}, 87(4):526--535, 2018.

\bibitem{WA11}
G.~Wiener and M.~Araya.
\newblock On planar hypohamiltonian graphs.
\newblock {\em Journal of Graph Theory}, 67(1):55--68, 2011.

\bibitem{Za12}
C.T. Zamfirescu.
\newblock Hypohamiltonian graphs and their crossing number.
\newblock {\em The Electronic Journal of Combinatorics}, 19(4), 2012.

\bibitem{Za15}
C.T. Zamfirescu.
\newblock On hypohamiltonian and almost hypohamiltonian graphs.
\newblock {\em Journal of Graph Theory}, 79(1):63--81, 2015.

\bibitem{ZZ07}
C.T. Zamfirescu and T.I. Zamfirescu.
\newblock A planar hypohamiltonian graph with 48 vertices.
\newblock {\em Journal of Graph Theory}, 55(4):338--342, 2007.

\end{thebibliography}

\end{document}